\documentclass[11pt,reqno]{amsart}%
\usepackage[a4paper,margin=3cm]{geometry}%
\usepackage{amssymb,latexsym,lmodern}
\usepackage[utf8]{inputenc}%
\usepackage[T1]{fontenc}%
\usepackage[unicode=true,pagebackref]{hyperref}%
\hypersetup{pdfborder=0 0 0}%
\title[Extremal eigenvalues of empirical covariance matrices]{On the
  convergence of the extremal eigenvalues of empirical covariance matrices
  with dependence}
\dedicatory{Dedicated to Alain Pajor \& Nicole Tomczak-Jaegermann}
\author{Djalil Chafaï}%
\address[DC]{CEREMADE, Université Paris-Dauphine, PSL, IUF, Paris, France.}%
\author{Konstantin Tikhomirov}%
\address[KT]{Dep.\ of Math.\ and Stat.\ Sciences, University of Alberta, Edmonton, Canada.}%
\theoremstyle{plain}%
\newtheorem{theorem}{Theorem}[section]%
\newtheorem{definition}[theorem]{Definition}%
\newtheorem{proposition}[theorem]{Proposition}%
\newtheorem{corollary}[theorem]{Corollary}%
\newtheorem{lemma}[theorem]{Lemma}%
\newcommand{\FIXME}[1]{%
  $\clubsuit$\marginpar{%
    \ifodd\value{page}\flushleft\else\flushright\fi%
      {\fbox{#1}}}}%
\newcommand{\dE}{{\mathbb{E}}}%
\newcommand{\dN}{{\mathbb{N}}}%
\newcommand{\dP}{{\mathbb{P}}}%
\newcommand{\dR}{{\mathbb{R}}}%
\newcommand{\dX}{{\mathbb{X}}}%
\newcommand{\card}{\mathrm{card}}%
\newcommand{\diag}{\mathrm{diag}}%
\newcommand{\proj}{\mathrm{P}}%
\newcommand{\rank}{\mathrm{rank}}%
\newcommand{\supp}{\mathrm{supp}}%
\newcommand{\trace}{\mathrm{tr}}%
\newcommand{\veps}{\varepsilon}%
\newcommand{\ABS}[1]{{\left|#1\right|}}%
\newcommand{\DOT}[1]{{\langle#1\rangle}}%
\newcommand{\IND}{{\mathbf{1}}}%
\newcommand{\NRM}[1]{{\left\Vert#1\right\Vert}}%
\newcommand{\PAR}[1]{{\left(#1\right)}}%
\numberwithin{equation}{section}
\begin{document}

\begin{abstract}
  Consider a sample of a centered random vector with unit covariance matrix.
  We show that under certain regularity assumptions, and up to a natural
  scaling, the smallest and the largest eigenvalues of the empirical
  covariance matrix converge, when the dimension and the sample size both tend
  to infinity, to the left and right edges of the Marchenko--Pastur
  distribution. The assumptions are related to tails of norms of orthogonal
  projections. They cover isotropic log-concave random vectors as well as
  random vectors with i.i.d.\ coordinates with almost optimal moment
  conditions. The method is a refinement of the rank one update approach used
  by Srivastava and Vershynin to produce non-asymptotic quantitative
  estimates. In other words we provide a new proof of the Bai and Yin theorem
  using basic tools from probability theory and linear algebra, together with
  a new extension of this theorem to random matrices with dependent entries.
\end{abstract}

\maketitle

{\small\tableofcontents}

\section{Introduction}

Let $\dN=\{1,2,\ldots\}$ and let ${(X_n)}_{n\in\dN}$ be a sequence of random
vectors where for each $n\in\dN$ the random vector $X_n$ takes values in
$\dR^n$, is centered, with unit covariance (isotropy):
\begin{equation}\label{eq:Xn}
  \dE(X_n)=0\quad\text{and}\quad\dE(X_n\otimes X_n)=I_n
\end{equation}
where $I_n$ is the $n\times n$ identity matrix. Let $(m_n)_{n\in\dN}$ be a
sequence in $\dN$ such that
\begin{equation}\label{eq:mn}
  0<\underline{\rho}:=\liminf_{n\to\infty} \frac{n}{m_n}
  \le\limsup_{n\to\infty} \frac{n}{m_n}=:\overline{\rho}<\infty.
\end{equation}
For every $n\in\dN$, let $X_n^{(1)},\ldots,X_n^{(m_n)}$ be i.i.d.\ copies of
$X_n$. Their empirical covariance matrix is the $n\times n$ symmetric
positive semidefinite random matrix
\begin{equation}\label{eq:Sigman}
\widehat\Sigma_n
:=\frac{1}{m_n}\sum_{k=1}^{m_n} X_n^{(k)}\otimes X_n^{(k)}.
\end{equation}
If $\dX_n$ denotes the $m_n\times n$ rectangular random matrix with i.i.d.\ rows
$X_n^{(1)},\ldots,X_n^{(m_n)}$ then
\[
\widehat\Sigma_n=\frac{1}{m_n}\dX_n^\top \dX_n.
\]
Note that $\dE\widehat\Sigma_n=\dE(X_n\otimes X_n)=I_n$. For
convenience we define the random matrix
\begin{equation}\label{eq:A}
  A_n:=m_n\widehat\Sigma_n%
  =\dX_n^\top \dX_n%
  =\sum_{k=1}^{m_n}X_n^{(k)}\otimes X_n^{(k)}.
\end{equation}
The eigenvalues of $A_n$ are squares of the singular values of
$\dX_n$, and in particular
\[
\lambda_{\max}(A_n)
=s_{\max}(\dX_n)^2
=\max_{\|x\|=1}\|\dX_nx\|^2
=\|\dX_n\|_{\mathrm{2\to2}}^2.
\]
When $m_n\geq n$ then the smallest eigenvalue of $A_n$ satisfies
\[
\lambda_{\min}(A_n)
=s_{\min}(\dX_n)^2
=\min_{\|x\|=1}\|\dX_n x\|^2
=\|\dX_n^{-1}\|_{\mathrm{2\to2}}^{-2}
\]
where the last formula holds only when $\dX_n$ is invertible (impossible if
$m_n<n$). Above and in the sequel we denote by
$\|x\|=(x_1^2+\cdots+x_n^2)^{1/2}$ the Euclidean norm of $x\in\dR^n$.

If ${((X_n)_k)}_{1\leq k\leq n,n\geq1}$ are i.i.d.\ standard Gaussians, then
the law of the random matrix $\widehat\Sigma_n$ is known as the real Wishart
law, and constitutes a sort of a matrix version of the $\chi^2(n)$ law. The law
of the eigenvalues of $\widehat\Sigma_n$ is then called the Laguerre
Orthogonal Ensemble, a Boltzmann--Gibbs measure with density on
$\{\lambda\in[0,\infty)^m:\lambda_1\geq\cdots\geq\lambda_n\}$ proportional to
\[  
\lambda\mapsto%
\exp\Bigr(-\frac{m_n}{2}\sum_{k=1}^n\lambda_k%
+\frac{m_n-n+1}{2}\sum_{k=1}^n\log(\lambda_n)%
+\sum_{i<j}\log\ABS{\lambda_i-\lambda_j}\Bigr),
\]
see \cite[(7.4.2) p.~192]{MR2808038}. This exactly solvable Gaussian model
allows to deduce sharp asymptotics for the empirical spectral distribution as
well as for the extremal eigenvalues of $\widehat\Sigma_n$. The famous
Marchenko--Pastur theorem \cite{MR2808038,MR0208649} states that if
\[
\frac{n}{m_n}\underset{n\to\infty}{\longrightarrow}\rho%
\quad\text{with}\quad%
\rho\in(0,\infty)%
\]
then almost surely as $n\to\infty$, the empirical spectral distribution of
$\widehat\Sigma_n$ tends weakly with respect to continuous and bounded test functions
to a non-random distribution, namely
\begin{equation}\label{eq:MP}
a.s.\quad
\frac{1}{m_n}\sum_{k=1}^n\delta_{\lambda_k(\widehat\Sigma_n)}
\overset{\mathcal{C}_b}{\underset{n\to\infty}{\longrightarrow}}
\mu_\rho
\end{equation}
where $\mu_\rho$ is the Marchenko--Pastur distibution 
on $[a^-,a^+]$ with $a^\pm=(1\pm\sqrt{\rho})^2$ given by
\[
\mu_\rho(dx)
=\frac{\rho-1}{\rho}\IND_{\rho>1}\delta_0
+\frac{\sqrt{(a^+-x)(x-a^-)}}{\rho 2\pi x}\,\IND_{[a^-,a^+]}(x)dx.
\]
It is a mixture between a Dirac mass at point $0$ and an absolutely continuous
law. The atom at $0$ disappears when $\rho\leq1$ and is a reminiscence of the
rank of $\widehat\Sigma_n$. The asymptotic phenomenon \eqref{eq:MP} holds
beyond the Gaussian case. In particular it was shown by Pajor and Pastur
\cite{MR2539559} that it holds if for every $n\in\dN$ the distribution of the
isotropic random vector $X_n$ is log-concave. Recall that a probability
measure $\mu$ on $\dR^n$ with density $\varphi$ is log-concave when
$\varphi=e^{-V}$ with $V$ convex, see \cite{MR0388475,MR3185453}.
Log-concavity allows some kind of geometric dependence but imposes
sub-exponential tails. The asympotitic phenomenon \eqref{eq:MP} also holds if
${((X_n)_k)}_{1\leq k\leq n,n\geq1}$ are i.i.d.\ with finite second moment
\cite{MR2567175,MR2808038}. An extension to various other models can be found
in Bai and Zhou \cite{MR2411613}, Pastur and Shcherbina \cite{MR2808038},
Adamczak \cite{MR2820070}, and Adamczak and Chafaï \cite{MR3359233}.

The weak convergence \eqref{eq:MP} does not provide much information at the
edge on the behavior of the extremal atoms, and what one can actually extract
from \eqref{eq:MP} is that
\begin{equation}\label{eq:MPedge}
  a.s.\quad 
  \limsup_{n\to\infty}\lambda_{\min}(\widehat\Sigma_n)
  \leq
  (1-\sqrt{\rho})^2
  \leq
  (1+\sqrt{\rho})^2
  \leq
  \liminf_{n\to\infty}\lambda_{\max}(\widehat\Sigma_n)
\end{equation}
where the first inequality is considered only in the case where $m_n\geq n$.
If ${((X_n)_k)}_{1\leq k\leq n,n\geq1}$ are i.i.d.\ with finite fourth moment
then it was shown by Bai and Yin \cite{MR958213,MR950344,MR1235416} using
combinatorics that the convergence in \eqref{eq:MPedge} holds:
\begin{equation}\label{eq:BY}
  a.s.\quad 
  (1-\sqrt{\rho})^2
  =
  \lim_{n\to\infty}\lambda_{\min}(\widehat\Sigma_n)
  \leq
  \lim_{n\to\infty}\lambda_{\max}(\widehat\Sigma_n)
  =
  (1+\sqrt{\rho})^2,  
\end{equation}
where the first equality is considered only in the case where $m_n\geq n$. The
convergence of the largest eigenvalue in the right hand side of \eqref{eq:BY}
does not take place if ${((X_n)_k)}_{1\leq k\leq n,n\geq1}$ are i.i.d.\ with
infinite fourth moment, see \cite{MR2567175}. It was understood recently that
the convergence of the smallest eigenvalue in the left hand side of
\eqref{eq:BY} holds actually as soon as ${((X_n)_k)}_{1\leq k\leq n,n\geq1}$
are i.i.d.\ with finite second moment, see Tikhomirov
\cite{tikhomirov-asymptotic}.

An analytic proof of \eqref{eq:BY} based on the resolvent is also available,
and we refer to Bordenave \cite{bordenave} for the i.i.d.\ case, and to Bai
and Silverstein \cite{MR1617051}, Pillai and Yin \cite{pillai-yin}, and
Richard and Guionnet \cite{richard-guionnet} for more sophisticated models
still not including the case in which the law of $X_n$ is log-concave for
every $n\in\dN$. 
Note that the analytic approach was also used for various models of random
matrices by Haagerup and Thorbj{\o}rnsen \cite{MR2183281}, Schultz
\cite{MR2117954}, and by Capitaine and Donati-Martin \cite{MR2317545}.

The study of quantitative high dimensional non-asymptotic properties of the
smallest and of the largest eigenvalues of empirical covariance matrices was
the subject of an intense line of research in the recent years; in connection
with log-concavity, see for instance Adamczak, Litvak, Pajor, and
Tomczak-Jaegermann \cite{MR2601042,MR2769907}, Rudelson and Vershynin
\cite{MR2827856}, Koltchinskii and Mendelson \cite{koltchinskii-mendelson},
and references therein.

Non-asymptotic estimates for \eqref{eq:BY} were obtained by Srivastava and
Vershynin \cite{srivastava-vershynin} using a rank one update strategy which
takes advantage of the decomposition \eqref{eq:A}. This approach is an
elementary interplay between probability and linear algebra, which is
remarkably neither analytic nor combinatorial. The outcome is that with high
probability
\begin{equation}\label{eq:SV}
\Bigr(1-c^-\sqrt{\frac{n}{m_n}}\,\Bigr)^2
\leq
\lambda_{\min}(\widehat\Sigma_n)
\leq
\lambda_{\max}(\widehat\Sigma_n)
\leq \Bigr(1+c^+\sqrt{\frac{n}{m_n}}\,\Bigr)^2,
\end{equation}
where the first inequality is considered only in the case where $m_n\geq n$.
Here $c^\pm>0$ are constants, and thus one cannot deduce \eqref{eq:BY} from
\eqref{eq:SV}. The approach of Srivastava and Vershynin is a randomization of the
spectral sparsification method developed by Batson, Spielman, and Srivastava \cite{MR2780071,MR3029269};
the idea of using rank one updates can also be found in the
works of Benaych-Georges and Nadakuditi \cite{MR2944410}. This approach requires the
control of tails of norms of projections of $X_n$, a condition which is
satisfied by log-concave distributions thanks to the thin-shell phenomenon.
This condition is also satisfied if ${((X_n)_k)}_{1\leq k\leq n,n\geq1}$ are
i.i.d.\ with a finite moment of order $2+\veps$ for the lower bound on the
smallest eigenvalue in \eqref{eq:BY} and i.i.d.\ with a finite moment of order
$4+\veps$ for the upper bound on the largest eigenvalue in \eqref{eq:BY}. This
method was also used recently by Yaskov \cite{yaskov,yaskov-sharp}.

Our main results below lie between the original Bai--Yin theorem and the more
recent work of Srivastava and Vershynin. Our contribution is to show that the
non-asymptotic approach of Srivastava and Vershynin is indeed suitable to
prove and extend the sharp Bai--Yin theorem, which is an asymptotic result,
under fairly general asumptions on tails of norms of projections of $X_n$,
which allow heavy tailed i.i.d.\ as well as log-concavity! When the
coordinates of $X_n$ are i.i.d.\ our approach reaches the (almost) optimal
moment condition: finite second moment for the smallest eigenvalue and finite
fourth moment for the largest.

Our results are based on the following tail conditions on the norm of
projections.

\begin{definition}[Weak Tail Projection property $\mathrm{(WTP)}$]
  Let $X_n$, $n\in\dN$, be as in \eqref{eq:Xn}. We say that the \emph{Weak
    Tail Projection} $\mathrm{(WTP)}$ property holds when the following
  is true:
  \begin{itemize}
  \item[\textbf{(a)}] The family ${(\langle
      X_{n},y\rangle^2)}_{n\in\dN,y\in S^{n-1}}$ is uniformly
    integrable, in other words
      \begin{equation*}\label{eq:WTP-a}\tag{WTP-a} 
        \lim_{M\to\infty}\sup_{\substack{n\in\dN\\y\in S^{n-1}}}
        \dE\bigl(%
        \langle X_n,y\rangle^2\IND_{\{\langle X_n,y\rangle^2\geq M\}}%
        \bigr)=0,
      \end{equation*}
      where $S^{n-1}:=\{y\in\dR^n:\|y\|=1\}$ denotes the unit sphere of
      $\dR^n$;
    \item[\textbf{(b)}] There exist two functions $f:\dN\to[0,1]$ and
      $g:\dN\to\dR_+$ such that $f(r)\to0$ and $g(r)\to0$ as $r\to\infty$ and
      for every $n\in\dN$ and any orthogonal projection $\proj:\dR^n\to\dR^n$
      with $\proj\neq0$,
  \begin{equation*}\label{eq:WTP-b}\tag{WTP-b} 
    \dP\PAR{\|\proj X_n\|^2-\rank \proj\ge f(\rank \proj)\rank \proj}
    \leq g(\rank \proj).
  \end{equation*}
  This can be less formally written as $\dP\bigl(\NRM{\proj X_n}^2-r\geq
  o(r)\bigr)\leq o(1)$ where $r:=\rank\proj$ and where the ``$o$'' are with
  respect to $r$ and are uniform in $n$.
\end{itemize}
\end{definition}

\begin{definition}[Strong Tail Projection property (STP)]
  Let $X_n$, $n\in\dN$, be as in \eqref{eq:Xn}. We say the \emph{Strong Tail
    Projection} (STP) property holds when there exist $f:\dN\to[0,1]$ and
  $g:\dN\to\dR_+$ such that $f(r)\to0$ and $g(r)\to0$ as $r\to\infty$, and for
  every $n\in\dN$, for any orthogonal projection $\proj:\dR^n\to\dR^n$ with
  $\proj\neq0$, for any real $t\geq f(\rank \proj)\rank \proj$ we have
  \begin{equation*}\label{eq:STP}\tag{STP} 
    \dP\PAR{\|\proj X_n\|^2-\rank \proj\ge t} %
    \le \frac{g(\rank \proj)\rank \proj}{t^2}.
  \end{equation*}
  This can be less formally written as $\dP(\NRM{\proj X_n}^2-r\geq t)\leq
  o(r)t^{-2}$ where $t\geq o(r)$ and $r:=\rank\proj$ and where the ``$o$'' are
  with respect to $r$ and are uniform in $n$.
\end{definition}

Note that $\dE(\|\proj X_n\|^2)=\rank\proj$ since $X_n$ is isotropic. The
properties $\mathrm{(WTP)}$ and \eqref{eq:STP} were inspired by the ``strong
regularity assumption'' used by Srivastava and Vershynin
\cite{srivastava-vershynin}. They are specially designed to obtain a sharp
Bai--Yin type asymptotic result in the i.i.d.\ case with (almost) optimal
moment assumptions, as well as in the log-concave case.

It is easy to see that \eqref{eq:STP} implies that for any $\veps>0$,
the $4-\veps$ moments of $1$-dimensional marginals of $X_n$'s are
uniformly bounded; in particular, in this case the vectors ${(X_n)}_{n\in\dN}$
satisfy condition \eqref{eq:WTP-a}. Next, condition \eqref{eq:WTP-b} is
clearly weaker than \eqref{eq:STP}; thus, \eqref{eq:STP} implies
$\mathrm{(WTP)}$, hence the qualifiers ``Strong'' and ``Weak'' for these
properties.

Proposition \ref{pr:logconc} below, proved in Section \ref{se:pr:logconc+iid},
implies that if ${(X_n)}_{n\in\dN}$ is as in \eqref{eq:Xn} and if $X_n$ is
log-concave for every $n\in\dN$ then properties $\mathrm{(WTP)}$ and
\eqref{eq:STP} are satisfied. It is a consequence of the thin-shell and
sub-exponential tails phenomena for these distributions.

\begin{proposition}[Log-concave random vectors]
  \label{pr:logconc}
  Let $X_n$, $n\in\dN$, be as in \eqref{eq:Xn}, and suppose that the centered
  isotropic random vector $X_n$ is log-concave for any $n\in\dN$. Then a
  stronger form of \eqref{eq:STP} holds with
  \[
  \{\NRM{\proj X_n}^2-\rank\proj\geq t\}%
  \quad\text{replaced by}\quad%
  \{|\NRM{\proj X_n}^2-\rank \proj|\geq t\}.
  \]
\end{proposition}

The next proposition implies that
if ${((X_n)_k)}_{1\leq k\leq n,n\in\dN}$ are i.i.d.\ then property
$\mathrm{(WTP)}$ holds, and if moreover these i.i.d.\ random variables have
finite fourth moment then property \eqref{eq:STP} holds.

\begin{proposition}[Random vectors with i.i.d.\ coordinates]
  \label{pr:iid}
  Let $X_n$, $n\in\dN$, be as in \eqref{eq:Xn}, and suppose that the
  coordinates ${((X_n)_k)}_{1\leq k\leq n,n\in\dN}$ are i.i.d.\
  distributed as a real random variable $\xi$ with zero mean and unit
  variance. Then, denoting $r:=\rank\proj$,
  \begin{itemize}
  \item a stronger version of $\mathrm{(WTP)}$ holds, with, in
    \eqref{eq:WTP-b},
    \[
    \{\|\proj X_n\|^2-r\geq f(r)r\}%
    \quad\text{replaced by}\quad%
    \{|\|\proj X_n\|^2-r|\geq f(r)r\};
    \]
  \item if moreover $\dE(\xi^4)<\infty$ then a stronger version of
    \eqref{eq:STP} holds, with
    \[
    \{\|\proj X_n\|^2-r\geq t\}%
    \quad\text{replaced by}\quad%
    \{|\|\proj X_n\|^2-r|\geq t\}.
    \]
  \end{itemize}
\end{proposition}

Our main results are stated in the following couple of theorems and corollaries.

\begin{theorem}[Smallest eigenvalue]\label{th:min} %
  Let $X_n$, $m_n$, and $A_n$, $n\in\dN$, be as in \eqref{eq:Xn},
  \eqref{eq:mn}, and \eqref{eq:A} respectively. If $\overline{\rho}<1$ (in
  particular $m_n> n$ for $n\gg1$), and if \eqref{eq:WTP-a} and
  \eqref{eq:WTP-b} are satisfied then
  \[
  \liminf_{n\to\infty}\frac{\dE(\lambda_{\min}(A_n))}{(\sqrt{m_n}-\sqrt{n})^2}\ge 1.
  \]
\end{theorem}

Theorem \ref{th:min} is proved in Section \ref{se:min}. 

Combining Theorem \ref{th:min} with Proposition \ref{pr:logconc} and
Proposition \ref{pr:iid}, we obtain the following corollary. The second
part, which is the Bai--Yin edge convergence \eqref{eq:BY} of the smallest
eigenvalue in probability, is obtained by combining the first part of the
corollary with the Marchenko--Pastur bound \eqref{eq:MPedge} for the smallest
eigenvalue.

\begin{corollary}[Smallest eigenvalue convergence]\label{co:max}
  Let $X_n$, $m_n$, $\widehat\Sigma_n$, and $A_n$, $n\in\dN$, be as in
  \eqref{eq:Xn}, \eqref{eq:mn}, \eqref{eq:Sigman}, and \eqref{eq:A}
  respectively. If $\overline{\rho}<1$ (in particular $m_n> n$ for $n\gg1$)
  and if the centered isotropic random vector $X_n$ is log-concave for every
  $n\in\dN$ or if ${((X_n)_k)}_{1\leq k\leq n,n\in\dN}$ are i.i.d.\ then
  \[
  \liminf_{n\to\infty}\frac{\dE(\lambda_{\min}(A_n))}{(\sqrt{m_n}-\sqrt{n})^2}\ge 1.
  \]
  If additionally $\lim_{n\to\infty}\frac{n}{m_n}=\rho$ with $\rho\in(0,1)$,
  in other words $\underline{\rho}=\overline{\rho}\in(0,1)$, then
  \[
  \lambda_{\min}(\widehat\Sigma_n)
  \overset{\dP}{\underset{n\to\infty}{\longrightarrow}}
  (1-\sqrt{\rho})^2.
  \]
\end{corollary}

\begin{theorem}[Largest eigenvalue]\label{th:max}
  Let $X_n$, $m_n$, and $A_n$, $n\in\dN$, be as in \eqref{eq:Xn},
  \eqref{eq:mn}, and \eqref{eq:A} respectively. If \eqref{eq:STP} holds then
  \[
  \limsup_{n\to\infty}\frac{\dE(\lambda_{\max}(A))}{(\sqrt{m_n}+\sqrt{n})^2}\le 1.
  \]
\end{theorem}

Theorem \ref{th:max} is proved in Section \ref{se:max}.

Combining Theorem \ref{th:max} with Proposition \ref{pr:logconc} and
Proposition \ref{pr:iid}, we obtain the following statement (again, for
the second part we use the Marchenko--Pastur law):

\begin{corollary}[Largest eigenvalue convergence]\label{co:min}
  Let $X_n$, $m_n$, $\widehat\Sigma_n$ and $A_n$, $,n\in\dN$, be as in
  \eqref{eq:Xn}, \eqref{eq:mn}, \eqref{eq:Sigman}, and \eqref{eq:A}
  respectively. If the centered isotropic random vector $X_n$ is log-concave
  for every $n\in\dN$ or if ${((X_n)_k)}_{1\leq k\leq n,n\in\dN}$ are i.i.d.\
  with finite $4$th moment then
  \[
  \limsup_{n\to\infty}\frac{\dE(\lambda_{\max}(A_n))}{(\sqrt{m_n}+\sqrt{n})^2}\le 1.
  \]
  If additionally $\lim_{n\to\infty}\frac{n}{m_n}=\rho$ with $\rho\in(0,\infty)$
  in other words $\underline{\rho}=\overline{\rho}\in(0,\infty)$, then
  \[
  \lambda_{\max}(\widehat\Sigma_n)
  \overset{\dP}{\underset{n\to\infty}{\longrightarrow}}
  (1+\sqrt{\rho})^2.
  \]
\end{corollary}

\subsection*{Outline of the argument and novelty}

Assume we are given a random matrix
\[
A=\sum_{k=1}^{m}X^{(k)}\otimes X^{(k)},
\]
where $X^{(k)}$ are i.i.d.\ isotropic random vectors. As we already mentioned
above, the key ingredient in estimating the extremal eigenvalues of $A$ is the
following rank one update formula known as the Sherman--Morrison formula:
\[
(M+xx^\top)^{-1}=M^{-1}-\frac{M^{-1}xx^\top M^{-1}}{1+x^\top M^{-1}x},
\]
which is valid for any non-singular $n\times n$ matrix $M$ and a vector $x$
with $1+x^\top M^{-1}x\neq 0$. Using the above identity and assuming that $M$ is
symmetric, the restriction on $\dR$ of the Cauchy--Stieltjes transform of the
spectral distribution of $M+xx^T$, which is defined as an appropriately scaled
trace of $(u-M-xx^\top)^{-1}$, $u\in\dR$, can be in a simple way expressed in
terms of the Cauchy--Stieltjes transform of $M$. To be more precise, setting
\[
A^{(0)}:=0%
\quad\text{and}\quad%
A^{(k)}:=A^{(k-1)}+X^{(k)}\otimes X^{(k)},\quad k=1,\dots,m,
\]
we get, for any $k=1,\dots,m$ and any $u\in\dR$,
\[
\trace (u-A^{(k)})^{-1}=\trace (u-A^{(k-1)})^{-1}
+\frac{{X^{(k)}}^\top(u-A^{(k-1)})^{-2}X^{(k)}}{1-{X^{(k)}}^\top(u-A^{(k-1)})^{-1}X^{(k)}}.
\]
A crucial observation, made already in \cite{MR2780071} and further developed
in \cite{srivastava-vershynin} is that the Cauchy--Stieltjes transform on the
real line can be efficiently used to control the extreme eigenvalues of the
matrix. The basic idea is, starting from some fixed $u_0\neq 0$, to define
inductively a sequence of {\it random} numbers $u_k$ ($k\leq m$) in such a way
that all $u_k$'s stay on the same side of the spectra of $A^{(k)}$'s, at the
same time not departing too far from the spectrum. Then the expectation of the
corresponding extreme eigenvalue of $A=A^{(m)}$ can be estimated by $\dE u_m$. The
increments $u_k-u_{k-1}$ are defined with help of the last formula as certain
functions of $A^{(k-1)}$, $u_{k-1}$ and $X^{(k)}$, and their expectations are
controlled using the information on the distribution of $X^{(k)}$ as well as
certain induction hypothesis. At this level, the approach used in the present
paper is similar to \cite{srivastava-vershynin}.

On the other hand, as our result is asymptotically sharp and covers the
i.i.d.\ case with almost optimal moment conditions, the technical part of our
argument differs significantly from \cite{srivastava-vershynin}. In
particular, we introduce the ``regularity shifts'', which are designed in such
a way that $u_k$'s stay ``sufficiently far'' from the spectrum of $A^{(k)}$'s,
which guarantees validity of certain concentration inequalities, whereas at
the same time not departing ``too far'' so that one still gets a satisfactory
estimate of the expectation of the spectral edges. The shifts (which we denote by $\delta^k_R$ and $\Delta^k_R$)
are defined in Sections~\ref{se:proof of thmin} and~\ref{se:proof of thmax} (see, in
particular, lemmas~\ref{le:lb:deltaC} and~\ref{le:deltaC}).

Let us emphasize once more that the proofs we obtain are much simpler and
shorter than the original combinatorial approach of Bai--Yin and the
analytic approach based on the resolvent (more precisely the Cauchy--Stieltjes
transform on the complex plane outside the real axis), while the class of
distributions covered in our paper is much larger. In
particular, Theorem~\ref{th:min} of the present paper essentially recovers a
recent result \cite{tikhomirov-asymptotic}.
In must be noted, however, that in our paper we replace
convergence almost surely with the weaker convergence in probability. 

\subsection*{Discussion and extensions}

In this note we restrict our analysis to random vectors with real coordinates
because we think that this is simply more adapted to geometric dependence such
as log-concavity. It is likely that the method remains valid for random
vectors with complex entries. The Bai--Yin theorem is also available for
random symmetric matrices (which are the sum of rank two updates which are no
longer positive semidefinite) but it is unclear to us if one can adapt the
method to this situation. One can ask in another direction if the method
remains usable for non-white population covariance matrices, and for the
so-called information plus noise covariance matrices, two models studied at
the edge by Bai and Silverstein \cite{MR1617051,MR2930382} among others. One
can ask finally if the method allows to extract at least the scaling of the
Tracy--Widom fluctuation at the edge. For the Tracy--Widom fluctuation at the
edge of empirical covariance matrices we refer to Johansson
\cite{MR1737991}, %
Johnstone \cite{MR1863961}, %
Borodin and Forrester \cite{MR1986402}, %
Soshnikov \cite{MR1933444}, %
Péché \cite{MR2475670}, %
Feldheim and Sodin \cite{MR2647136}, %
Pillai and Yin \cite{pillai-yin}, %
and references therein. It was shown by Lee and Yin \cite{MR3161313} that for
centered Wigner matrices the finiteness of the fourth moment is more than
enough for the Tracy--Widom fluctuation of the largest eigenvalue. One can ask
the same for the largest eigenvalue of the empirical covariance matrix of
random vectors with i.i.d.\ entries, and one can additionally ask if a finite
second moment is enough for the Tracy--Widom fluctuation of the smallest
eigenvalue.

\subsection*{Structure}

The rest of the article is structured as follows. Section
\ref{se:pr:logconc+iid} provides the proof of Proposition \ref{pr:logconc} and
Proposition \ref{pr:iid}. In Sections~\ref{se:min} and~\ref{se:max} we
prove Theorem \ref{th:min} and Theorem \ref{th:max},
respectively.

\subsection*{Notations}

We set $\NRM{v}:=\sqrt{v_1+\cdots+v_n^2}$ and
$\NRM{v}_\infty:=\max(|v_1|,\ldots,|v_n|)$ for any vector $v\in\dR^n$. We
denote $\NRM{f}_\infty:=\sup_{x\in S}|f(x)|$ for any function $f:S\to\dR$. We
often use the notation $|S|:=\card(S)$ for a set $S$. Further, we denote by
\[
\lambda_{\max}(M):=\lambda_1(M)\geq\cdots\geq\lambda_n(M)=:\lambda_{\min}(M)
\]
the eigenvalues of a symmetric $n\times n$ matrix $M\in\mathcal{M}_n(\dR)$.
We denote by $I_n$ the $n\times n$ identity matrix. For any real number $u$ we
sometimes abridge $uI_n$ into $u$ in matrix expressions such as in
$M-uI_n=M-u$. 




\section{Proof of Proposition \ref{pr:logconc} and Proposition \ref{pr:iid}}
\label{se:pr:logconc+iid}

\begin{proof}[Proof of Proposition \ref{pr:logconc}]
  Assume $X$ is a centered isotropic log-concave vector in $\dR^n$ and $\proj:\dR^n\to\dR^n$ is a non-zero
  orthogonal projection.
  The random vector $\proj X$ is log-concave with mean zero and covariance
  matrix $\proj\proj^\top=\proj^2=\proj$. Restricted to the image of $\proj$,
  the vector $Y=\proj X$ is log-concave with covariance matrix $I_r$
  where $r=\rank \proj$. The so-called thin-shell phenomenon \cite{MR1997580}
  states that ``most'' of the distribution of $Y$ is supported in a thin-shell
  around the Euclidean sphere of radius $\sqrt{r}$. Quantitative estimates
  have been obtained notably by Klartag \cite{MR2311626}, Fleury
  \cite{MR2652173}, Guédon and Milman \cite{MR2846382}, see also Guédon
  \cite{MR3178607}. On the other hand, it is also known that the tail of the
  of norm of $Y$ is sub-exponenial, see Paouris \cite{MR2276533}, and also
  Admaczak et al \cite{MR3150710}. The following inequality, taken from
  \cite[Theorem 1.1]{MR2846382}, captures both phenomena: there exist
  absolute constants $c,C\in(0,\infty)$ such that for any real $u\geq0$,
  \[
  \dP\PAR{\ABS{\NRM{Y}-\sqrt{r}}\geq u\sqrt{r}}
  \leq C\exp\bigr(-c\sqrt{r}\min(u,u^3)\bigr).
  \]
  This is more than enough for our needs. Namely, let $\beta\in(0,1/20)$, and
  let $u=u(r)\in(0,\infty)$ and $\alpha=\alpha(r)\in(0,1)$ be such that
  $\alpha\geq(1+r)^{-\beta}$ and
  $u\geq\max((1+r)^{-\beta},2\alpha/(1-\alpha^2))$. Note that
  $2\alpha/(1-\alpha^2)\in(0,1)$ when $\alpha\in(0,\sqrt{2}-1)$, and that
  $2\alpha/(1-\alpha^2)\to0$ as $\alpha\to0$. Now, using the inequality
  $\exp(-2t)\leq t^{-4}$ for $t>0$, we get, if $\alpha u\in(0,1]$,
  \begin{align*}
  \dP(\NRM{Y}^2-r\geq u^2r)
  &=\dP(\NRM{Y}\geq\sqrt{r+u^2r}) \\
  &\leq\dP(\NRM{Y}\geq\sqrt{r}+\alpha u\sqrt{r}) \\
  &\leq C\exp\bigr(-c\sqrt{r}(\alpha u)^3\bigr) \\
  &\leq \frac{2^4C}{c^4(u^2r)^2\alpha^{12}u^8} = \frac{o(r)}{(u^2r)^2},
  \end{align*}
  while if $\alpha u\in[1,\infty)$ we get 
  \begin{align*}
  \dP(\NRM{Y}^2-r\geq u^2r)
  \leq \cdots
  \leq C\exp\bigr(-c\sqrt{r}\alpha u\bigr) 
  \leq \frac{2^4C}{c^4(u^2r)^2\alpha^4} = \frac{o(r)}{(u^2r)^2}.
  \end{align*}
  Similarly, for an arbitrary $\beta\in(0,1/20)$, let $u=u(r)\in(0,\infty)$
  and $\alpha=\alpha(r)\in(0,1)$ be such that $\alpha\geq(1+r)^{-\beta}$ and
  $u\geq\max((1+r)^{-\beta},2\alpha/(1+\alpha^2))$.
  If $u>1$ then necessarily
  $$\dP(\NRM{Y}^2-r\leq -u^2r)=0.$$
  Otherwise, $\alpha u\in(0,1]$ and using again the inequality
  $\exp(-2t)\leq t^{-4}$ for $t>0$, we get
  \begin{align*}
  \dP(\NRM{Y}^2-r\leq -u^2r)
  &\leq\dP(\NRM{Y}\leq\sqrt{r-u^2r}) \\
  &\leq\dP(\NRM{Y}\leq\sqrt{r}-\alpha u\sqrt{r}) \\
  &\leq C\exp\bigr(-c\sqrt{r}(\alpha u)^3\bigr) \\
  &\leq \frac{2^4C}{c^4(u^2r)^2\alpha^{12}u^{8}} %
  = \frac{o(r)}{(u^2r)^2}.
  \end{align*}

  Thus, we obtain $\dP\bigr(\bigr|\NRM{Y}^2-r\bigr|\geq t\bigr)\leq o(r)t^{-2}$ for
  $t\geq o(r)$ as expected.
\end{proof}

\begin{proof}[Proof of Proposition \ref{pr:iid}] \ %
  \begin{itemize}
  \item\emph{Proof of the first part (uniform integrability
      \eqref{eq:WTP-a}).} Recall that we are given a random variable $\xi$
    with zero mean and unit variance and that for every $n\in\dN$ the
    coordinates of $X_n$ are independent copies of $\xi$. We want to show that
    \[
    \lim_{M\to\infty}\sup_{\substack{n\in\dN\\y\in S^{n-1}}}%
    \dE\bigl(\langle X_n,y\rangle^2\IND_{\{\langle X_n,y\rangle^2\geq M\}}\bigr)=0.
    \]
    For every $x\in\dR^n$, we define $f_n(x):=\DOT{X_n,x}$. Clearly,
    $\dE(f_n^2(x))=\NRM{x}^2$ since $X_n$ is isotropic. Let us start with some
    comments to understand the problem. The random variables
    ${(f_n^2(y))}_{n\in\dN,y\in S^{n-1}}$ belong to the unit sphere of $L^1$.
    If $\xi$ has finite fourth moment $B:=\dE(\xi^4)<\infty$ then by expanding
    and using isotropy we get $\dE(f_n^4(y))\leq\max(B,3)$ which shows that
    the family ${(f_n^2(y))}_{n\in\dN,y\in S^{n-1}}$ is bounded in $L^2$ and
    thus uniformly integrable. How to proceed when $\xi$ does not have a
    finite fourth moment? If $y$ belongs to the canonical basis of $\dR^n$
    then $f_n(y)$ is distributed as $\xi$ and has thus the same integrability.
    On the other hand, if $y$ is far from being sparse, say
    $y=(n^{-1/2},\ldots,n^{-1/2})$, then $f_n(y)$ is distributed as
    $n^{-1/2}(\xi_1+\cdots+\xi_n)$ where the $\xi_i$'s are independent copies
    of $\xi$, which is close in distribution to the standard Gaussian law
    $\mathcal{N}(0,1)$ by the Central Limit Theorem (CLT).

    We will use the following uniform quantitative CLT. Even though it probably
    exists somewhere in the literature, we provide a short proof for
    convenience. It can probably also be proved using the classical Fourier
    analytic smoothing inequality \cite[equation (3.13) XVI.3
    p.~538]{MR0270403} which is the basis of many uniform CLT estimates.

    \begin{lemma}[Uniform quantitative CLT]\label{le:UQCLT}
      Let $\xi$ be a random variable with zero mean and unit variance and
      let $\Phi$ be the cumulative distribution of the standard real Gaussian
      law $\mathcal{N}(0,1)$ of zero mean and unit variance. For any $\veps>0$
      there exists $\delta>0$ depending only on $\veps$ and the law of $\xi$
      with the following property: if $n\in\dN$ and $y\in S^{n-1}$ is such
      that $\NRM{y}_\infty\leq\delta$, then the cumulative distribution
      function $F_n$ of $\sum_{i=1}^ny_i\xi_i$ satisfies
      \[
      \NRM{F_n-\Phi}_\infty\leq\veps.
      \]
    \end{lemma}

    \begin{proof}[Proof of Lemma \ref{le:UQCLT}]
      To prove the lemma, let us assume the contrary. Then there exists
      $\veps>0$ and a sequence ${(y_m)}_{m\geq1}$ in $\ell^2(\dN)$ such that
      $\NRM{y_m}=1$ and $\NRM{y_m}_\infty\leq1/m$ for every $m\in\dN$, and such
      that $\NRM{F_m-\Phi}_\infty>\veps$ where $F_m$ is the cumulative
      distribution function of $S_m:=\sum_{i=1}^\infty y_{m,i}\xi_i$. Let
      $\varphi_m$ be the characteristic function of $S_m$. We have
      $\varphi_m(t)=\prod_{i=1}^\infty\varphi(y_{m,i}t)$ where $\varphi$ is
      the characteristic function of $\xi$. Fix any $t\in\dR$. By assumption
      on $\xi$, we get $\varphi(t)=1-\frac{t^2}{2}+o(t^2)$. Hence, using the identities
      $\NRM{y_m}^2:=\sum_{i=1}^\infty y_{m,i}^2=1$ and
      $\NRM{y_m}_\infty:=\max_{i\in\dN}|y_{m,i}|\leq 1/m$ together with the
      formula (valid for $m\to\infty$):
      \[
      \varphi_m(t)%
      =\prod_{i=1}^\infty\PAR{1-\frac{y_{m,i}^2t^2}{2}+o((y_{m,i}^2t^2))}%
      =\exp\PAR{-\frac{t^2}{2}\sum_{i=1}^\infty y_{m,i}^2+\sum_{i=1}^\infty o(y_{m,i}^2t^2)},
      \]
      we get $\lim_{m\to\infty}\varphi_m(t)=e^{-t^2/2}$. By the Lévy theorem
      for characteristic functions, it follows that $F_m\to\Phi$ pointwise as
      $m\to\infty$, which yields $S_m\to\mathcal{N}(0,1)$ weakly as
      $m\to\infty$, contradicting to our initial assumption.
    \end{proof}
    
    Let us continue with the proof of the uniform integrability. 
    Since $\xi^2\in L^1$ we get, by dominated convergence, 
    \[
    h(M):=\dE(\xi^2\IND_{\{\xi^2\geq M\}})%
    \underset{M\to\infty}{\longrightarrow}0.
    \]
    Let $\veps>0$, and let $\delta>0$ be defined from $\veps$ and the law of
    $\xi$ as in the above lemma (we can, of course, assume that $\delta\to 0$
    with $\veps\to 0$). Let $M>0$ and $n\in\dN$ and $y\in S^{n-1}$.
    Let us write $y=w+z$ where $w_i:=y_i\IND_{\ABS{y_i}\leq\delta^2}$ and
    $z_i:=y_i\IND_{\ABS{y_i}>\delta^2}$ for any $i=1,\ldots,n$. Then it is
    easily seen that
    \begin{itemize}
    \item[$\bullet$] $\supp(z)\cap\supp(w)=\varnothing$;
    \item[$\bullet$] $\NRM{w}\leq1$ and $\NRM{w}_\infty\leq\delta^2$;
    \item[$\bullet$] $\NRM{z}\leq1$ and $|\supp(z)|\leq 1/\delta^{4}$,
    \end{itemize}    
    where $\supp(x):=\{i:x_i\neq0\}$. Now we have
    $f_n^2(y)\leq2(f_n^2(w)+f_n^2(z))$ and
    \begin{align*}
      \dE(f^2_n(y)\IND_{\{f^2_n(y)\geq M\}}) %
      &\leq%
      2\dE((f_n^2(w)+f_n^2(z))%
      (\IND_{\{f_n^2(w)\geq\frac{M}{4}\}}+\IND_{\{f_n^2(z)\geq\frac{M}{4}\}}))\\
      &\leq 2\NRM{z}^2\dP\Bigr(f_n^2(w)\geq{\scriptstyle\frac{1}{4}}M\Bigr)%
      +2\NRM{w}^2\dP\Bigr(f_n^2(z)\geq{\scriptstyle\frac{1}{4}}M\Bigr)\\
      &\quad+2\dE(f_n^2(z)\IND_{\{f_n^2(z)\geq\frac{M}{4}\}})%
      +2\dE(f_n^2(w)\IND_{\{f_n^2(w)\geq\frac{M}{4}\}})\\
      &=:(*)+(**)\\
      &\quad+(***)+(****).
    \end{align*}
    Now, by Markov's inequality
    \[
    (*)+(**)
    \leq \frac{16}{M}.
    \]
    Second, using that $f_n^2(x)\leq
    \|\proj_x X_n\|^2$, valid for any $x\in\dR^n$ with $\|x\|\leq
    1$ and orthogonal projection $\proj_x$ onto its support, we obtain
    \begin{align*}
      (***)
      &\leq 2%
      \dE\Bigr(\Bigr(\sum_{i:z_i\neq0}\xi_i^2\Bigr)%
      \sum_{i:z_i\neq0}\IND_{\{\xi_i^2\geq\frac{\delta^8M}{4}\}}\Bigr)\\
      &= 2%
      \dE\Bigr(\sum_{i,j\in\supp(z)}%
      \xi_i^2\IND_{\{\xi_j^2\geq\frac{\delta^8M}{4}\}}\Bigr)\\
      &\leq 2%
      \Bigr(\frac{1}{\delta^8}\frac{4}{\delta^8M}%
      +\frac{h(\frac{\delta^8M}{4})}{\delta^4}\Bigr)%
      =\frac{8}{\delta^{16}M}+\frac{2h(\frac{\delta^8M}{4})}{\delta^{4}},
    \end{align*}
    where in the third line we used Markov's inequality.
    Third, we write
    \[
    (****)=2\dE(f_n^2(w))-2\dE(f_n^2(w)\IND_{\{f_n^2(w)<\frac{M}{4}\}}).
    \]
    Now if $\NRM{w}\leq\delta$ then $(****)\leq 2\delta^2$. Suppose in contrast
    that $\NRM{w}>\delta$, and denote $w_*:=w/\NRM{w}$. Then
    $\NRM{w_*}_\infty\leq\delta$, and therefore, by the CLT of Lemma
    \ref{le:UQCLT}, the distribution of $f_n(w_*)$ is $\veps$-close to
    $\mathcal{N}(0,1)$, and in particular, there exist $M_*(\veps)>0$ and
    $\rho(\veps)>0$ depending {\it only} on $\veps$ such that $\rho(\veps)\to0$ as
    $\veps\to0$ and
    \[
    \dE(f_n^2(w_*)\IND_{\{f_n^2(w_*)<M_*(\veps)\}})\geq 1-\rho(\veps).
    \]
    Thus, having in mind that $f_n^2(w_*)=f_n^2(w)/\|w\|^2$,
    we get
    \[
    (****)\leq2\|w\|^2\rho(\veps)\leq 2\rho(\veps)
    \]
    provided that $M\geq4M_*(\veps)$. So, in all cases,
    as long as $M$ is sufficiently large,
    \[
    (****)\leq 2\delta^2+2\rho(\veps).
    \]
    
    Finally, take any $M>0$ such that 
    $M\geq\max(16/\veps,8/(\delta^{16}\veps),4M_*(\veps))$
    and such that $2h(\delta^8M/4)/\delta^{14}\leq\veps$; then
    the desired result follows from
    \[
    (*)+(**)+(***)+(****)\leq4\veps+2\delta^2+2\rho(\veps).
    \]

  \item\emph{Proof of the first part (improved \eqref{eq:WTP-b}).}
    As before, we assume that $\xi$ is a random variable with zero mean and unit variance,
    and denote by $(\xi_i)$ a sequence of i.i.d.\ copies of $\xi$. Let us
    recall a kind of the weak law of large numbers for weighted sums, taken from
    \cite[Lemma 5]{tikhomirov-asymptotic}, which can be seen as a consequence
    of Lévy's continuity theorem for characteristic functions: if
    ${(\eta_i)}_{i\in I}$ is a sequence (finite or infinite) of i.i.d.\
    real random variables with zero mean then for every $\veps>0$ there exists
    $\delta>0$ which may depend on the law of the random variables with the
    following property: for every deterministic sequence ${(t_i)}_{i\in I}$ in
    $[0,\infty)$ such that $\NRM{t}_1:=\sum_{i\in I}t_i=1$ and
    $\NRM{t}_\infty:=\max_{i\in I}t_i\leq\delta$, we have
    \[
    \dP\Bigr(\Bigr|\sum_{i\in I}t_i\eta_i\Bigr|>\veps\Bigr)<\veps.
    \]
    Setting $\eta_i:=\xi_i^2-1$, it follows that there exists $h:(0,1]\to(0,1]$ such that, given any
    $\veps>0$ and a sequence ${(t_i)}_{i\in I}$ in $[0,\infty)$ with
    $\NRM{t}_1=1$ and $\NRM{t}_\infty\leq h(\veps)$, we have
    \begin{equation}\label{eq:deviation0}
      \dP\Bigr(\Bigr|\sum_{i\in I}t_i\xi_i^2-1\Bigr|>\veps\Bigr)\leq \veps.
    \end{equation}
    Without loss of generality, one can take $h$ strictly monotone (i.e.\
    invertible).
    
    We now proceed similarly to \cite[Proposition~1.3]{srivastava-vershynin}.
    Fix $n\in\dN$ and let $\proj$ be a non-zero orthogonal projection of $\dR^n$ of rank
    $r$. Let $X={(\xi_k)}_{1\leq k\leq n}$ be a random vector of
    $\dR^n$ with $\xi_1,\ldots,\xi_n$ i.i.d.\ copies of $\xi$. We have
    \[
    \NRM{\proj X}^2=\DOT{X,\proj X}.
    \]
    Let us also denote the matrix of $\proj$ in the canonical basis as
    $\proj$. We have $P^2=P=P^\top$ and $\trace(\proj)=\rank(\proj)=r$. Let
    $\proj_0$ be the matrix obtained from $\proj$ by zeroing the diagonal. We
    have $\proj-\proj_0=\diag(\proj)$. A standard decoupling inequality proved
    in \cite{decoupling} (see also the book \cite{ MR1666908}) states that for
    any convex function $F$,
    \begin{equation}\label{eq:decoupling}
    \dE(F(\DOT{X,\proj_0X}))\leq \dE(F(4\DOT{X,\proj_0X'})),
    \end{equation}
    where $X'$ is an independent copy of $X$. In particular the choice
    $F(u)=u^2$ gives
    \[
    \dE(\DOT{X,\proj_0X}^2)\leq 16\,\dE(\DOT{X,\proj_0X'}^2).
    \]
    Now recall that if $Z$ is a random vector of $\dR^n$ with covariance
    matrix $\Gamma$ and if $B$ is a $n\times n$ matrix then
    $\dE(\DOT{Z,BZ})=\trace(\Gamma B^\top)$. Seeing $X$ and $X'$ as column
    vectors,
    \begin{align*}
      \dE(\DOT{X,\proj_0X'}^2)
      &
      =\dE(X'^\top \proj_0 XX^\top \proj_0X')\\
      &
      =\dE(X'^\top \proj_0 \dE(XX^\top) \proj_0X')\\
      &
      =\dE(X'^\top \proj_0^2X')\\
      &
      =\trace(\proj_0^2)\\
      &
      \leq 2\trace((\proj-\proj_0)^2)+2\trace(\proj^2)\\
      &
      \leq 2\trace(\proj^2)+2\trace(\proj^2)\\
      &
      = 4r.
    \end{align*}
    Therefore
    \begin{equation*}\label{eq:ep0}
      \dE\PAR{\DOT{X,\proj_0 X}^2}\leq 64r
    \end{equation*}
    and by Markov's inequality we get
    \begin{equation}\label{eq:deviation1}
      \dP(\ABS{\DOT{X,\proj_0X}}>r^{3/4})\leq\frac{64}{\sqrt{r}}.
    \end{equation}
    Next note that 
    \begin{equation*}\label{eq:diag}
      \DOT{X,(\proj-\proj_0)X} =\sum_{i=1}^n \proj_{i,i}\xi_i^2
    \end{equation*}
    with $0\leq P_{i,i}\leq 1$ and $\sum_{i=1}^n P_{i,i}=r$. Hence taking
    $t_i=P_{i,i}/r$ and $\veps:=h^{-1}(1/r)$ with $\veps:=1$ if $1/r$ is outside
    of the range of $h$, we get, using \eqref{eq:deviation0},
    \[
    \dP\PAR{\ABS{\DOT{X,(\proj-\proj_0)X}-r}>\veps r}\leq\veps.
    \]
    Finally, by combining with \eqref{eq:deviation1} we obtain
    \[
    \dP\PAR{\ABS{\DOT{X,\proj X}-r}>(\veps+r^{-1/4})r}\leq\veps+\frac{64}{\sqrt{r}}
    \]
    and this implies the desired result, namely
    \[
    \dP\PAR{\ABS{\DOT{X,\proj X}-r}>o(r)}= o(1).
    \]    
  \item\emph{Proof of the second part (improved \eqref{eq:STP}).} Let
    $\xi$ be a random variable with zero mean, unit variance and a finite fourth moment.
    Further, let $\proj=(\proj_{ij})$, $\proj_0$ and $r$ have the same meaning as before, and $X=(\xi_k)_{1\leq k\leq n}$,
    where $\xi_k$'s are i.i.d.\ copies of $\xi$.
    For any $u>0$ we have
    \begin{align*}
    \dP&\PAR{\ABS{\|\proj X\|^2-r}>u}\\
    &\le \dP\PAR{\ABS{\DOT{X,\proj_0 X}}>u/2}%
    +\dP\PAR{\ABS{\DOT{X,(\proj-\proj_0) X}-r}>u/2}\\
    &=:(*)+(**).
    \end{align*}
    Let us estimate the last two quantities $(*)$ and $(**)$ separately. In
    both cases we will compute selected moments and use Markov's inequality.
    Set $B:=\dE(\xi_k^4)$. Note that $B\geq1$ since $\dE\xi^2=1$.
    Since $\xi_k$'s are independent, we
    get, for any unit vector $y=(y_k)_{1\leq k\leq n}$,
    \begin{equation}\label{eq:fourth moment}
    \dE(\DOT{X,y}^4)
    =
    \sum_i\dE(\xi_i^4)y_i^2y_i^2
    +3\sum_{i\neq j}\dE(\xi_i^2)\dE(\xi_j^2)y_i^2y_j^2
    \leq \max(B,3).
    \end{equation}
    The decoupling inequality \eqref{eq:decoupling} with $F(u)=u^4$ gives
    \[
    \dE(\DOT{X,\proj_0X}^4)\leq 256\,\dE(\DOT{X,\proj_0X'}^4),
    \]
    where $X'$ is an independent copy of $X$.
    Next, in view of \eqref{eq:fourth moment}, we get
    \[
    \dE(\DOT{X,\proj_0X'}^4) %
    \leq\dE(\|\proj_0X'\|^4)\max_{\NRM{y}=1}\dE(\DOT{X,y}^4) %
    \leq \max(B,3)\dE(\|\proj_0X'\|^4).
    \]
    Since $X'$ has independent coordinates of zero mean and unit variance, we get
    \begin{align*}
    \dE(\|\proj_0X'\|^4) %
    &\leq 8\dE(\|\proj X'\|^4) %
    +8\dE(\|(\proj-\proj_0)X'\|^4)\\
    &\leq 8\dE\Bigl(\sum\limits_{i,j}\proj_{ij}\xi_i\xi_j\Bigr)^2+
    8\dE\Bigl(\sum\limits_{i=1}^n\xi_i^2{\proj_{ii}}^2\Bigr)^2\\
    &\leq 8\max(B,2)\sum\limits_{i,j}{\proj_{ij}}^2+8\sum\limits_{i\neq j}\proj_{ii}\proj_{jj}
    +8B\sum_{i,j}{\proj_{ii}}^2{\proj_{jj}}^2\\
    &\leq 8\max(B,2)r+8r^2 %
    +8Br^2\\
    &\leq 32Br^2.
    \end{align*}
    Hence, $\dE(\DOT{X,\proj_0X}^4)\leq (256Br)^2$, and applying Markov's
    inequality, we get the following bound for $(*)$:
    \[
    (*)\le \frac{(1024Br)^2}{u^4}.
    \]

    Let us turn to estimating $(**)$. We will use symmetrization, truncation,
    and concentration. Let $\widetilde X=(\widetilde \xi_k)_{1\leq k\leq n}$ be an
    independent copy of $X$. Note that
    \[
    \dE((\DOT{\widetilde X,(\proj-\proj_0) \widetilde X}-r)^2)
    =\dE\Bigl(\Bigl(\sum_{k=1}^n\proj_{kk}({\widetilde \xi_k}^2-1)\Bigr)^2\Bigr)
    \leq Br,
    \]
    so, using the independence of $X$ and $\widetilde X$ and applying Markov's
    inequality, we get
    \begin{align*}
    (**)&\leq 2\dP\PAR{\ABS{\DOT{X,(\proj-\proj_0) X}%
        -\DOT{\widetilde X,(\proj-\proj_0) \widetilde X}}>u/2-\sqrt{2Br}}\\
    &= 2\dP\Bigl(\Bigl|\sum_{k=1}^n%
    \proj_{kk}({\xi_k}^2-{\widetilde\xi_k}^2)\Bigr|>u/2-\sqrt{2Br}\Bigr).
    \end{align*}
    Clearly, the variables ${\xi_k}^2-{\widetilde\xi_k}^2$ ($1\leq k\leq n$) are
    symmetrically distributed, with the variance bounded from above by $2B$.
    Let $h:\dR_+\to\dR_+$ be a function defined as
    \[
    h(t):=\dE\PAR{({\xi_k}^2-{\widetilde\xi_k}^2)^2\chi_t},
    \]
    where $\chi_t$ denotes the indicator of the event
    $\{|{\xi_k}^2-{\widetilde\xi_k}^2|\ge t\}$. Clearly, $h(t)\to 0$ when $t$
    tends to infinity. Note that, by Hoeffding's inequality, we have
    \begin{align*}
    \dP\Bigl(\Bigl|%
    \sum_{k=1}^n\proj_{kk}({\xi_k}^2-{\widetilde\xi_k}^2)(1-\chi_{r^{1/4}})\Bigr|>u/4\Bigr)
    &\leq 2\exp\Bigl(-\frac{2u^2}{16\sqrt{r}}\bigl(4\sum\nolimits_{k}{\proj_{kk}}^2\bigr)^{-1}\Bigr)\\
    &\leq 2\exp\bigl(-u^2/(32r\sqrt{r})\bigr).
    \end{align*}
    On the other hand, since the random variable
    $({\xi_k}^2-{\widetilde\xi_k}^2)\chi_{r^{1/4}}$ has zero mean, we have
    \[
    \dE\Bigl(\Bigl(\sum_{k=1}^n\proj_{kk}({\xi_k}^2-{\widetilde\xi_k}^2)\chi_{r^{1/4}}\Bigr)^2\Bigr)
    \le r h(r^{1/4}).
    \]
    Applying Markov's inequality, we get for all $u>4\sqrt{2Br}$,
    \[
    \dP\Bigl(\Bigl|%
    \sum_{k=1}^n\proj_{kk}({\xi_k}^2-{\widetilde\xi_k}^2)\chi_{r^{1/4}}\Bigr|>u/4-\sqrt{2Br}\Bigr)%
    \le \frac{r h(r^{1/4})}{(u/4-\sqrt{2Br})^2}.
    \]
    Combining the estimates, we obtain
    \begin{align*}
    (**)&\leq
    2\dP\Bigl(\Bigl|%
    \sum_{k=1}^n\proj_{kk}({\xi_k}^2-{\widetilde\xi_k}^2)(1-\chi_{r^{1/4}})\Bigr|>u/4\Bigr)\\
    &\hspace{1cm}+
    2\dP\Bigl(\Bigl|%
    \sum_{k=1}^n\proj_{kk}({\xi_k}^2-{\widetilde\xi_k}^2)\chi_{r^{1/4}}\Bigr|>u/4-\sqrt{2Br}\Bigr)\\
    &\leq 4\exp\bigl(-u^2/(32r\sqrt{r})\bigr)+\frac{128r h(r^{1/4})}{u^2} %
    \quad\text{for all $u>8\sqrt{2Br}$}.
    \end{align*}
    Finally, grouping together the bounds for $(*)$ and $(**)$, we get for all
    $u>8\sqrt{2Br}$,
    \[
    \dP\PAR{\ABS{\|\proj X\|^2-r}>u} %
    \leq %
    \frac{(1024Br)^2}{u^4} %
    +4\exp\bigl(-u^2/(32r\sqrt{r})\bigr) %
    +\frac{128r h(r^{1/4})}{u^2}.
    \]
    Set $\alpha\in(0,1/6)$. For any $u>8\sqrt{2B}r^{1-\alpha}=o(r)$, using the
    inequality $e^{-2t}\leq 1/t^4$ for $t>0$, the right hand side of the
    last equation above is bounded above by
    \[
    \frac{r}{u^2}\frac{(1024 B)^2r}{u^2} \\
    +\frac{r}{u^2}4\frac{64^4r^{5}}{u^6} 
    +\frac{r}{u^2}128h(r^{1/4})
    =\frac{o(r)}{u^2}.
    \]
    This proves the desired result.
    We note that the proof can be shortened and simplified under the stronger
    assumption that $\dE(|\xi_1|^p)<\infty$ for some $p>4$, see also
    \cite[Proposition 1.3]{srivastava-vershynin} for thin-shell estimates in
    the same spirit.
  \end{itemize}
\end{proof}

\section{Proof of Theorem \ref{th:min}}
\label{se:min}

As in \cite{MR3029269,srivastava-vershynin}, for every
$t\in\dR\setminus\{\lambda_1(S),\ldots,\lambda_n(S)\}$, we set
\[
\underline{m}_S(t)
:= \trace((S-tI_n)^{-1})=\sum_{k=1}^n\frac{1}{\lambda_k(S)-t}
\]
The function $t\mapsto\underline{m}_S(t)$ is positive and strictly increasing
on $(-\infty,\lambda_n(S))$. We note that $n\underline{m}_S$ is the restriction on
$\dR$ of the Cauchy--Stieltjes transform of the empirical spectral distribution
of $S$. What is important to us is that $\underline{m}_S$ encodes as
singularities the eigenvalues of $S$, is monotone on $(-\infty,\lambda_n(S))$, and
behaves nicely under rank one updates of $S$.

\subsection{Feasible lower shift}\label{se:feas-lower-shift}

Let $A$ be an $n\times n$ positive semi-definite non-random matrix with
eigenvalues $\lambda_{\max}:=\lambda_1\ge\cdots\ge\lambda_n=:\lambda_{\min}\ge
0$ and a corresponding orthonormal basis of eigenvectors $(x_i)_{i=1}^n$, and
let $u<\lambda_{\min}$. Further, assume that
$x$ is a (non-random) vector in $\dR^n$. We are interested in
those numbers $\delta\ge 0$ that (deterministically) satisfy
\begin{equation}\label{eq:small-gen-conds}
\lambda_{\min}>u+\delta\quad\text{and}\quad
\underline{m}_{A+xx^\top}(u+\delta)\le\underline{m}_A(u).
\end{equation}
Following \cite{srivastava-vershynin}, any value of $\delta$ satisfying
\eqref{eq:small-gen-conds}, will be called \emph{a feasible lower shift} with
respect to $A$, $x$ and $u$. The following statement is taken from
\cite{srivastava-vershynin}; we provide its proof for reader's convenience.

\begin{lemma}[Feasible lower shift -- {\cite[Lemma~2.2]{srivastava-vershynin}}]
  \label{le:q1q2}
  Let $\delta\geq 0$ be such that $u+\delta<\lambda_{\min}$. Let us define
  \begin{align*}
    q_1(\delta)
    &:=x^\top(A-u-\delta)^{-1}x
    =\sum_{i=1}^n\frac{\DOT{x,x_i}^2}{\lambda_i-u-\delta}\\
    q_2(\delta)&:=
    \frac{x^\top(A-u-\delta)^{-2}x}{\trace((A-u-\delta)^{-2})} %
    = \Bigr(\sum_{i=1}^n(\lambda_i-u-\delta)^{-2}\Bigr)^{-1}
    \sum_{i=1}^n\frac{\langle x,x_i\rangle^2}{\bigl(\lambda_i-u-\delta\bigr)^2}.
  \end{align*}
  Then a sufficient condition for \eqref{eq:small-gen-conds} to be satisfied is
  $
  q_2(\delta)\ge \delta (1+q_1(\delta)).
  $
\end{lemma}
\begin{proof}  
  If $S$ is a symmetric matrix and if $x$ is a vector, and if both $S$ and
  $S+xx^\top$ are invertible, then $1+x^\top S^{-1} x\neq0$ since $x^\top
  S^{-1}x=-1$ gives $(S+xx^\top)S^{-1}x=0$. Moreover the inverse
  $(S+xx^\top)^{-1}$ of the rank one update $S+xx^\top$ of $S$ can be
  expressed as
  \begin{equation}\label{eq:sherman-morrison}
    (S+xx^\top)^{-1} = S^{-1}-\frac{S^{-1}xx^\top S^{-1}}{1+x^\top S^{-1}x}.
  \end{equation}
  This is known as the Sherman--Morrison formula. This allows to write
  \begin{align*}
    \underline{m}_{A+xx^\top}(u+\delta)
    &=\trace((A-u-\delta+xx^\top)^{-1}) \\
    &=\trace((A-u-\delta)^{-1})
    -\frac{x^\top(A-u-\delta)^{-2}x}{1+x^\top(A-u-\delta)^{-1}x}\\
    &=\underline{m}_A(u)
    +\trace((A-u-\delta)^{-1}-(A-u)^{-1})
    -\frac{x^\top(A-u-\delta)^{-2}x}{1+x^\top(A-u-\delta)^{-1}x}.
  \end{align*}
  Now since $A-(u+\delta)I_n$ is positive definite, it follows that
  \[
  \delta(A-u-\delta)^{-2}-((A-u-\delta)^{-1}-(A-u)^{-1}) 
  \]
  is positive definite, and therefore
  \[
  \underline{m}_{A+xx^\top}(u+\delta)-\underline{m}_A(u)
  \leq
  \delta\trace((A-u-\delta)^{-2})
  -\frac{x^\top(A-u-\delta)^{-2}x}{1+x^\top(A-u-\delta)^{-1}x}.
  \]
  Finally it can be checked that the right hand side is $\leq0$ if
  $\delta(1+q_1(\delta))-q_2(\delta)\leq0$.
\end{proof}

\begin{lemma}[Construction of the feasible shift]\label{le:de}
Let $A$, $x$, $u$ and $q_1,q_2$ be as above, $\veps\in(0,1)$ and assume that
\[
\lambda_{\min}-u\geq 2/\veps^2.
\]
Then the quantity
\[
\delta:=\frac{(1-\veps)%
  \IND_{\{q_1(1/\veps)\leq(1+\veps)\underline{m}_A(u)+\veps\}}}
{(1+\veps)(1+\underline{m}_A(u))}
\Bigr(\sum_{i=1}^n(\lambda_i-u)^{-2}\Bigr)^{-1} \sum_{i=1}^n\frac{\langle
  x,x_i\rangle^2\IND_{\{\langle x,x_i\rangle^2\leq
    1/\veps\}}}{\bigl(\lambda_i-u\bigr)^2}
\]
satisfies $\delta\leq 1/\veps$ and is a feasible lower shift w.r.t.\ $A$, $x$ and $u$, i.e.\
$\lambda_{\min}>u+\delta$ and $q_2(\delta)\geq \delta(1+q_1(\delta))$.
\end{lemma}

\begin{proof}
  First, note that the condition $\lambda_{\min}-u\geq 2/\veps^2$
  immediately implies that 
  \[
  \veps (\lambda_i-u)^2-\frac{2(\lambda_i-u)}{\veps}+\frac{1}{\veps^2} %
  \geq 0, \quad i=1,\dots,n,
  \]
  which is in turn equivalent to the relation
  \begin{equation}\label{eq:def:delta:epsilon}
    \frac{1}{(1-\veps)\bigl(\lambda_i-u\bigr)^2}%
    \geq \frac{1}{\bigl(\lambda_i-u-1/\veps\bigr)^2},\quad i=1,\dots,n.
  \end{equation}
  Now, let us return to $\delta$. The inequality $\delta\leq 1/\veps$
  follows directly from its definition. Next,
  \[
  \frac{(1+q_1(\delta))\IND_{\{q_1(1/\veps)\leq (1+\veps)\underline{m}_A(u)+\veps\}}}
  {(1+\veps)(1+\underline{m}_A(u))}\leq 1,
  \]
  so we get
  \[
  \delta(1+q_1(\delta))\leq (1-\veps)q_2(0)\leq q_2(\delta),
  \]
  where the last inequality comes from \eqref{eq:def:delta:epsilon} and
  the definition of $q_2$.
\end{proof}

\subsection{Randomization and control of expectations}

Let, as before, $A$ be an $n\times n$ non-random positive semidefinite matrix with
eigenvalues $\lambda_1\ge\dots\ge\lambda_n\ge 0$, and let $u<\lambda_n$.
We define a (random) quantity $\delta$ as in
Lemma~\ref{le:de}, replacing the fixed vector $x$ with a random isotropic
vector $X$.

\begin{lemma}\label{le:lb:q_1:control}
  Let $\veps\in\bigl(0,12^{-3}\bigr)$, $n\geq 12/\veps^4$,
  and let $X$ be a random isotropic vector in $\dR^n$ such that
  \[
  \dP\PAR{\|\proj X\|^2-\rank \proj\geq \veps\rank \proj/6}\leq \veps^4
  \]
  for any non-zero orthogonal projection $\proj$ of rank at least
  $\veps^{11}n/72$. Further, let the non-random matrix $A$, the numbers
  $(\lambda_i)_{i\leq n}$ and vectors $(x_i)_{i\leq n}$ be as above, and
  $u\in\dR$ be such that $\lambda_{\min}-u\geq
  6\veps^{-2}+\veps^{-1}$. Assume additionally that
  \begin{equation}\label{eq:lb:small:differ}
    \sum_{i=1}^n\frac{1}{(\lambda_i-u)(\lambda_i-u+1)}\leq \frac{1}{\veps n}.
  \end{equation}
  Then, with $q_1$ defined as in Lemma~\ref{le:q1q2} (with $X$ replacing the
  non-random vector $x$), we have
  \[
  \dP\bigl\{q_1(1/\veps)+1\geq (1+\veps)(1+\underline{m}_A(u))\bigr\} %
  \leq 4\veps^2.
  \]
\end{lemma}

\begin{proof}
  First, note that the lower bound on $\lambda_{\min}-u$ implies that
  $q_1(1/\veps)\leq (1+\veps/6)q_1(0)$ (deterministically). Let us
  split the index set $\{1,\dots,n\}$ into several subsets in the following
  way: First, let $I:=\{i\leq n:\,\lambda_i-u\leq \veps^4 n/12\}$. Next,
  we set 
  \[
  J:=\{i\leq n:\,\lambda_i-u\geq 6n/\veps^3\},
  \]
  so that
  \[
  \{1,\dots,n\}\setminus (I\cup J)%
  =\bigl\{i\leq n:\,\veps^4 n/12< \lambda_i-u<6n/\veps^3\bigr\}.
  \]
  Note that, by the choice of $\veps$, we have
  $\exp(1/(12\veps))\geq 72/\veps^7$. Hence, the interval
  $(\ln(\veps^4n/12),\ln(6n/\veps^3))$ can be partitioned into
  $\lfloor\veps^{-2}\rfloor$ subintervals ${\mathcal S}_k$ ($k\leq
  \veps^{-2}$) of length at most $\veps/6$ each. Then we let
  \[
  I_k:=\bigl\{i\leq n:\,\ln(\lambda_i-u)\in {\mathcal S}_k\bigr\},%
  \quad k\leq \veps^{-2}.
  \]
  Obviously,
  \[
  q_1(0)%
  =\sum_{i\in I}\frac{\DOT{X,x_i}^2}{\lambda_i-u}%
  +\sum_{i\in J}\frac{\DOT{X,x_i}^2}{\lambda_i-u}%
  +\sum_{k\leq \veps^{-2}}\sum_{i\in I_k}\frac{\DOT{X,x_i}^2}{\lambda_i-u}%
  =(*)+(**)+(***).
  \]
  Let us estimate the three quantities separately.

  First, in view of the condition \eqref{eq:lb:small:differ}, the lower bound
  on $n$ and the definition of $I$, we have
  \[
  \dE(*)=\sum_{i\in I}\frac{1}{\lambda_i-u}\leq
  \sum_{i\in I}\frac{\veps^4n/12+1}{(\lambda_i-u)(\lambda_i-u+1)}
  \leq \frac{\veps^3}{6}.
  \]
  Hence, by Markov's inequality,
  \[
  \dP\bigl\{(*)\geq \veps/6\bigr\}\leq \veps^2.
  \]

  Similarly,
  \[
  \dE(**)=\sum_{i\in J}\frac{1}{\lambda_i-u}\leq \frac{\veps^3}{6},
  \]
  whence
  \[
  \dP\bigl\{(**)\geq \veps/6\bigr\}\leq \veps^2.
  \]

  Now, we consider the quantity $(***)$. First, assume that for some $k\leq
  \veps^{-2}$ we have $|I_k|\leq \veps^{11} n/72$. Since for every
  $i\in I_k$ we have $\lambda_i-u\geq \veps^4 n/12$, we obtain
  \[
  \dE\sum_{i\in I_k}\frac{\DOT{X,x_i}^2}{\lambda_i-u}\leq \frac{\veps^7}{6},
  \]
  whence, by Markov's inequality,
  \[
  \dP\Bigl\{\sum_{i\in I_k}\frac{\DOT{X,x_i}^2}{\lambda_i-u} %
  \geq \veps^3/6\Bigr\}\leq \veps^4.
  \]
  Now, if for some $k\leq \veps^{-2}$ we have $|I_k|>\veps^{11}
  n/72$, then, by the condition on projections and the definition of $I_k$,
  denoting by $P_k$ the orthogonal projection onto the span of $(x_i)_{i\in
    I_k}$, we obtain
  \[
  \dP\Bigl\{\sum_{i\in I_k}\frac{\DOT{X,x_i}^2}{\lambda_i-u}\geq
  \sum_{i\in
    I_k}\frac{\exp(\veps/6)(1+\veps/6)}{\lambda_i-u}\Bigr\} \leq
  \dP\bigl\{\|\proj_k X\|^2\geq (1+\veps/6)\rank \proj_k\bigr\}\leq
  \veps^4.
  \]

  Combining all the above estimates together, we get
  \begin{align*}
  \dP&\bigl\{q_1(1/\veps)%
  \geq (1+\veps/6)%
  \bigl(\veps/2+\exp(\veps/6)(1+\veps/6)\underline{m}_A(u)\bigr)\bigr\}\\
  &\leq\dP\bigl\{(*)+(**)+(***)%
  \geq \veps/6+\veps/6+\veps/6+
  \exp(\veps/6)(1+\veps/6)\underline{m}_A(u)\bigr\}\\
  &\leq 4\veps^2.
  \end{align*}
  It remains to note that
  \[
  (1+\veps/6)\bigl(\veps/2%
  +\exp(\veps/6)(1+\veps/6)\underline{m}_A(u)\bigr)+1%
  \leq (1+\veps)(1+\underline{m}_A(u)).
  \]
\end{proof}

\begin{lemma}[Control of $\dE\delta$]\label{le:Ed1}
  Let $\veps\in\bigl(0,12^{-3}\bigr)$, $n\geq 12/\veps^4$, and let
  $A$, $(\lambda_i)_{i\leq n}$, $(x_i)_{i\leq n}$, $X$ and $\delta$ be the
  same as in Lemma~\ref{le:lb:q_1:control} (and satisfy the same conditions).
  Assume additionally that $\dE(\langle X,x_i\rangle^2\IND_{\{\langle
    X,x_i\rangle^2\leq 1/\veps\}})\geq r$ ($i\leq n$) for some $r\in[0,1]$. Then
  \[
  \dE\delta\geq \frac{(1-\veps)r}{(1+\veps)(1+\underline{m}_A(u))}-4\veps.
  \]
\end{lemma}

\begin{proof}
  Since
  \[
  \delta=\frac{(1-\veps)(1-\IND_{\{q_1(1/\veps)>
      (1+\veps)\underline{m}_A(u)+\veps\}})}
  {(1+\veps)(1+\underline{m}_A(u))}
  \Bigr(\sum_{i=1}^n(\lambda_i-u)^{-2}\Bigr)^{-1} \sum_{i=1}^n\frac{\langle
    X,x_i\rangle^2\IND_{\{\langle X,x_i\rangle^2\leq
      1/\veps\}}}{\bigl(\lambda_i-u\bigr)^2},
  \]
  and in view of the bound $\delta\leq 1/\veps$, we get
  \begin{align*}
    \dE\delta &\geq \frac{(1-\veps)}
    {(1+\veps)(1+\underline{m}_A(u))}%
    \Bigl(\sum_{i=1}^n(\lambda_i-u)^{-2}\Bigr)^{-1}
    \sum_{i=1}^n\frac{%
      \dE\bigl(\langle X,x_i\rangle^2\IND_{\{\langle X,x_i\rangle^2\leq 1/\veps\}}\bigr)}%
    {\bigl(\lambda_i-u\bigr)^2}\\
    &-\veps^{-1}\dP\bigl\{q_1(1/\veps)>
    (1+\veps)\underline{m}_A(u)+\veps\bigr\}.
  \end{align*}
  Finally, applying Lemma~\ref{le:lb:q_1:control} to the last expression, we get the result.
\end{proof}

\subsection{Proof of Theorem~\ref{th:min}, completed}\label{se:proof of thmin}

Let $(X_n)_{n\in\dN}$ be as in the statement of the theorem. Without loss of
generality, we can assume that both functions $f,g:\dN\to\dR_+$ in the Weak
Tail Projection property \eqref{eq:WTP-b} are non-increasing. Additionally let
us define
\[
h(M):=\sup_{\substack{n\in\dN\\y\in S^{n-1}}}%
  \dE\bigl(\langle X_n,y\rangle^2\IND_{\{\langle X_n,y\rangle^2\geq M\}}\bigr),%
  \quad M\geq 0
\]
Note that \eqref{eq:WTP-a} gives $\lim_{M\to\infty}h(M)=0$.

Take any $\veps\in(0,12^{-3})$ and define $n_\veps$ as the smallest integer
greater than $12/\veps^4$ such that {\it (a)} $g(\veps^{11}n_\veps/72)\leq
\veps^4$ and $f(\veps^{11}n_\veps/72)\leq \veps/6$ and {\it (b)}
for all $n\geq n_\veps$ we have $(\sqrt{m_n}-\sqrt{n})^2\geq \varepsilon$ and $m_n/n\geq 3\overline{\rho}^{-1}/4+1/4$
(the latter implies $(\sqrt{m_n/n}-1)^{-1}\leq 2/(\overline{\rho}^{-1/2}-1)$). From now
on, we fix an $n\geq n_\veps$, let $m:=m_n$ and let $X^{(1)},\dots,X^{(m)}$ be
i.i.d.\ copies of $X_n$. We define
\[
A^{(0)}:=0%
\quad\text{and}\quad%
A^{(k)}:=A^{(k-1)}+X^{(k)}\otimes X^{(k)},\quad 1\leq k\leq m,
\]
so that $A_n=A^{(m)}$. Set
\[
u_0:=n-\sqrt{mn}
\]
and let $u_1,\dots,u_m$ be a collection of \emph{random} numbers defined
inductively as follows:
\[
u_k:=u_{k-1}+\delta^k-\delta_R^k,
\]
where $\delta^k$ is defined as $\delta$ in Lemma~\ref{le:de} (with $A^{(k-1)}$, $u_{k-1}$
and $X^{(k)}$ replacing $A$, $u$ and $x$, respectively) and {\it the regularity shift} $\delta_R^k$ is defined by
\[
\delta_R^k:=\min\Bigl\{\ell\in\{0,1,\ldots\}:\,%
\underline{m}_{A^{(k)}}(u_{k-1}+\delta^k-\ell)
-\underline{m}_{A^{(k)}}(u_{k-1}+\delta^k-\ell-1)\leq\frac{1}{\veps n}\Bigr\}.
\]
Note that lemmas~\ref{le:q1q2} and~\ref{le:de} imply that we have,
(deterministically) for all $k\ge 0$,
\[
\underline{m}_{A^{(k)}}(u_k)%
\leq\underline{m}_{A^{(0)}}(u_0)%
=\frac{n}{\sqrt{mn}-n}.
\]

The ultimate purpose of the shift $\delta_R^k$ is to guarantee relation \eqref{eq:lb:small:differ}
which was an important condition in proving the concentration lemma~\ref{le:lb:q_1:control}.
Of course, since $\delta_R^k$ moves $u_k$ away from the spectrum, one must make sure
that the cumulative impact of the shifts $\delta_R^k$ is small enough and does not destroy
the desired asymptotic estimate.

\begin{lemma}\label{le:lb:deltaC}
  With $\delta_R^k$ ($1\leq k\leq m$) defined above, the following holds
  deterministically:
  \[
  \sum_{k=1}^m\delta_R^k\leq \frac{\veps n^2}{\sqrt{mn}-n}.
  \]
\end{lemma}

\begin{proof}
  Take any admissible $k\geq 1$. The definition of $\delta_R^k$ immediately
  implies that for all $0\leq\ell<\delta_R^k$ we have
  \[
  \underline{m}_{A^{(k)}}(u_{k-1}+\delta^k-\ell)-
  \underline{m}_{A^{(k)}}(u_{k-1}+\delta^k-\ell-1)>\frac{1}{\veps n}.
  \]
  Hence,
  \[
  \underline{m}_{A^{(k)}}(u_k)=\underline{m}_{A^{(k)}}(u_{k-1}+\delta^k-\delta_R^k)\\
  \leq\underline{m}_{A^{(k)}}(u_{k-1}+\delta^k)-\frac{\delta_R^k}{\veps n}\\
  \leq\underline{m}_{A^{(k-1)}}(u_{k-1})-\frac{\delta_R^k}{\veps n}.
  \]
  Thus, $\underline{m}_{A^{(k)}}(u_k)\leq
  \underline{m}_{A^{(k-1)}}(u_{k-1})-\frac{\delta_R^k}{\veps n}$ for all
  $k\ge 1$, which, together with the relations
  $0<\underline{m}_{A^{(m)}}(u_m)$ and
  $\underline{m}_{A^{(0)}}(u_0)=\frac{n}{\sqrt{mn}-n}$ implies the result.
\end{proof}

Now, fix for a moment any $k\in\{1,\dots,m\}$ and let $\mathcal{F}_{k-1}$ be
the $\sigma$-algebra generated by the vectors $X^{(1)},\dots,X^{(k-1)}$, with the convention
$\mathcal{F}_0=\{\varnothing,\Omega\}$. We
will first estimate the conditional expectation
$\dE(\delta^k\,|\,\mathcal{F}_{k-1})$.

Note that, by the definition of $u_{k-1}$, we have (deterministically)
\[
\underline{m}_{A^{(k-1)}}(u_{k-1})-\underline{m}_{A^{(k-1)}}(u_{k-1}-1)%
\leq \frac{1}{\veps n}
\]
(the above relation holds for $k>1$ in view of the definition of $\delta^{k-1}_R$ and
for $k=1$ ~--- because of the definition of $n_\varepsilon$).
Together with the lower bound on $n$ (which implies the conditions on
orthogonal projections assumed in Lemma~\ref{le:Ed1}) and the condition
\[
\underline{m}_{A^{(k-1)}}(u_{k-1})%
\leq \underline{m}_{A^{(0)}}(u_{0})%
=\frac{n}{\sqrt{mn}-n},
\]
we get from Lemma~\ref{le:Ed1} that
\[
\dE(\delta^k\,|\,\mathcal{F}_{k-1})%
\geq
\frac{(1-\veps)(1-h(1/\veps))}{(1+\veps)(1+\underline{m}_{A^{(0)}}(u_{0}))}-4\veps
=\frac{(1-\veps)(1-h(1/\veps))}{1+\veps}\Bigl(1-\sqrt{\frac{n}{m}}\Bigr)-4\veps.
\]
Hence, by the definition of $u_k$'s and Lemma~\ref{le:lb:deltaC}, we obtain
\[
\dE u_m\geq u_0+\frac{(1-\veps)(1-h(1/\veps))}{1+\veps}\bigl(m-\sqrt{mn}\bigr)-4\veps m
-\frac{\veps n^2}{\sqrt{mn}-n}.
\]
Since $u_m<\lambda_{\min}(A_n)$ (deterministically), we get from the above relation
\begin{align*}
  \dE\lambda_{\min}(A_n)%
  &\geq (\sqrt{m}-\sqrt{n})^2-
  (3\veps+h(1/\veps))\bigl(m-\sqrt{mn}\bigr)-4\veps m
  -\frac{\veps n^2}{\sqrt{mn}-n}\\
  &\geq (\sqrt{m}-\sqrt{n})^2-7\veps
  m-h(1/\veps)m-\frac{2\veps n}{\overline\rho^{-1/2}-1}.
\end{align*}

Since the above estimate holds for arbitrarily small $\veps$, and having
in mind that $h(1/\veps)\to 0$ with $\varepsilon\to0$, we get the desired result.

\section{Proof of Theorem \ref{th:max}}
\label{se:max}

Following \cite{MR3029269,srivastava-vershynin}, for arbitrary $n\times n$ symmetric matrix $S$ and
every $t\in\dR\setminus\{\lambda_1(S),\ldots,\lambda_n(S)\}$, we set
\[
\overline{m}_S(t)
:= \trace((tI_n-S)^{-1})=\sum_{k=1}^n\frac{1}{t-\lambda_k(S)}.
\]
The function $t\mapsto\overline{m}_S(t)=-\underline{m}_S(t)$ is positive and
strictly decreasing on $(\lambda_1(S),+\infty)$.

\subsection{Feasible upper shift}\label{ss:feasible-up-shift}

Let $A$ be an $n\times n$ positive semi-definite non-random matrix with
eigenvalues $\lambda_{\max}:=\lambda_1\ge\cdots\ge\lambda_n=:\lambda_{\min}\ge
0$ and a corresponding orthonormal basis of eigenvectors $(x_i)_{i=1}^n$, and
let $u>\lambda_1$ be such that $\overline{m}_A(u)<1$. Further, assume that $x$
is a (non-random) vector in $\dR^n$. In this section, we consider
those numbers $\Delta\ge 0$ that (deterministically) satisfy
\begin{equation}\label{eq:largest-gen-conds}
  \lambda_{\max}(A+xx^\top)<u+\Delta%
  \quad\text{and}\quad%
  \overline{m}_{A+xx^\top}(u+\Delta)\le\overline{m}_A(u).
\end{equation}
Following \cite{srivastava-vershynin}, any value of $\Delta$ satisfying
\eqref{eq:largest-gen-conds}, will be called \emph{a feasible upper shift}
with respect to $A$, $x$ and $u$. The following statement is taken from
\cite{srivastava-vershynin}; we provide its proof for completeness.

\begin{lemma}[Feasible upper shift -- {\cite[Lemma~3.3]{srivastava-vershynin}}]
  \label{le:Q1Q2}
  Let $u>\lambda_1$ and $\overline{m}_A(u)<1$.
  For any $\Delta>0$, define
  \begin{align*}
    Q_1(\Delta)
    &:=x^\top(u+\Delta-A)^{-1}x %
    =\sum_{i=1}^n\frac{\langle x,x_i\rangle^2}{u+\Delta-\lambda_i} \\
    Q_2(\Delta)
    &:= \frac{x^\top(u+\Delta-A)^{-2}x}{\overline{m}_A(u)-\overline{m}_A(u+\Delta)} %
    =\bigl(\overline{m}_A(u)-\overline{m}_A(u+\Delta)\bigr)^{-1}
    \sum_{i=1}^n\frac{\langle x,x_i\rangle^2}{\bigl(u+\Delta-\lambda_i\bigr)^2}.
  \end{align*}
  Then a sufficient condition for \eqref{eq:largest-gen-conds} to be satisfied
  is
  \[
  Q_1(\Delta)<1\quad\text{and}\quad Q_2(\Delta)\le 1-Q_1(\Delta).
  \]
\end{lemma}

\begin{proof}
  We can assume without loss of generality that $x\neq 0$, so that
  $Q_2(\Delta)>0$. The Sherman--Morrison formula \eqref{eq:sherman-morrison}
  gives
  \begin{align*}
    \overline{m}_{A+xx^\top}(u+\Delta)%
    &=\trace((u+\Delta)I_n-A-xx^\top)^{-1}\\
    &=\trace((u+\Delta)I_n-A)^{-1}
    +\frac{\bigl\|((u+\Delta)I_n-A)^{-1}x\bigr\|^2}%
    {1-x^\top\bigl((u+\Delta)I_n-A\bigr)^{-1}x}\\
    &=\overline{m}_A(u+\Delta) %
    +\Bigl(1-\sum_{i=1}^n\frac{\langle
      x,x_i\rangle^2}{u+\Delta-\lambda_i}\Bigr)^{-1}%
    \sum_{i=1}^n\frac{\langle
      x,x_i\rangle^2}{\bigl(u+\Delta-\lambda_i\bigr)^2}.
  \end{align*}
  Hence, we have
  $\overline{m}_{A+xx^\top}(u+\Delta)\le \overline{m}_{A}(u)$ if and only if
  \[
  \Bigl(1-\sum_{i=1}^n\frac{\langle x,x_i\rangle^2}{u+\Delta-\lambda_i}\Bigr)^{-1}
  \sum_{i=1}^n\frac{\langle x,x_i\rangle^2}{\bigl(u+\Delta-\lambda_i\bigr)^2}\le 
  \overline{m}_A(u)-\overline{m}_A(u+\Delta).
  \]
  But the latter inequality clearly holds if $Q_1(\Delta)<1$ and
  $Q_2(\Delta)\leq1-Q_1(\Delta)$.
  
  Next, the rank one matrix
  \[
  (u+\Delta-A)^{-1/2}x((u+\Delta-A)^{-1/2}x)^\top
  =(u+\Delta-A)^{-1/2}xx^\top(u+\Delta-A)^{-1/2}
  \]
  has eigenvalues $0$ and $\|(u+\Delta-A)^{-1/2}x\|^2$. But
  the condition $Q_1(\Delta)<1$ implies
  $\|(u+\Delta-A)^{-1/2}x\|^2=x^\top(u+\Delta-A)^{-1}x<1$. Therefore the
  matrix
  \[
  I_n-(u+\Delta-A)^{-1/2}xx^\top(u+\Delta-A)^{-1/2}
  \]
  is positive definite, implying that
  \[
  (u+\Delta-A)^{1/2}I_n(u+\Delta-A)^{1/2}-xx^\top%
  =(u+\Delta)I_n-(A+xx^\top)
  \]
  is positive definite (recall that for any two positive definite matrices $S$ and $T$, the matrix $S^{1/2}TS^{1/2}$ is positive definite). Hence, $u+\Delta>\lambda_{\max}(A+xx^\top)$.
\end{proof}

Following \cite{srivastava-vershynin}, we will treat the quantities $Q_1$ and
$Q_2$ separately: for a specially chosen number $\tau\in(0,1)$, we will take
$\Delta_1$ such that $Q_1(\Delta_1)\le\tau$, and $\Delta_2$ such that
$Q_2(\Delta_2)\le 1-\tau$. Then, in view of monotonicity of $Q_1$ and $Q_2$,
the sum $\Delta_1+\Delta_2$ will be a feasible upper shift. In our case,
$\tau$ shall be close to $\overline{m}_A(u)$.

Let us introduce ``level sets'' $I_j$ as follows:
\begin{equation}\label{eq:def level sets}
I_j:=\bigl\{i\le n:\, 4^{j-1}\le u-\lambda_i< 4^j\bigr\},\quad j\in\dN.
\end{equation}
Note that the condition $\overline{m}_A(u)<1$ immediately implies that
$|I_j|<4^j$ for all $j\geq 1$. Moreover, if we additionally assume that
$\max_{\ell\in\dN}\frac{|I_\ell|}{16^\ell}$ is sufficiently small then the
following estimate of the ratio $\frac{|I_j|}{4^j}$ is valid:
\begin{lemma}\label{le:prop Ij}
Let the sets $I_j$ be as above and assume that for some $\veps\in(0,1]$ we have
\[
\max_{\ell\in\dN}\frac{|I_\ell|}{16^\ell}\leq \frac{1}{\veps n}.
\]
Then
\[
\frac{|I_j|}{4^j}\leq \sqrt{\frac{|I_j|}{\veps n}},\;\;j\in\dN.
\]
\end{lemma}
\begin{proof}
Fix any natural $j$ such that $I_j\neq\varnothing$. Then, obviously,
\[
\frac{|I_j|}{4^j}\leq \frac{4^j}{\veps n}.
\]
Fix for a moment any $\alpha>0$. If $|I_j|\geq \alpha 4^j$ then
\[
\frac{|I_j|}{4^j}\leq \frac{\alpha^{-1} |I_j|}{\veps n}.
\]
Otherwise, $\frac{|I_j|}{4^j}\leq \alpha$. It remains to choose
$\alpha:=\sqrt{\frac{|I_j|}{\veps n}}$.
\end{proof}

For every $j\geq 1$, denote
\begin{equation}\label{eq:def hj}
h_j:=\sum_{i\in I_j}\langle x,x_i\rangle^2-|I_j|.
\end{equation}

\begin{lemma}[Definition of $\Delta_1$]\label{le:Q1}
  Let $A$, $x$ and $u$ be as above, with $u>\lambda_{\max}$ and $\overline{m}_A(u)<1$,
  and let $\veps\in(0,1/4]$ be a real parameter.
  Define a number $\Delta_1$ as follows:
  \[
  \Delta_1:=\Delta_1'+\sum_{j=1}^{\lfloor\log_4(n/\veps^2)\rfloor}\Delta_{1,j},
  \]
  where
  \[
  \Delta_1':=\veps^{-1/2}\|x\|^2\IND_{\{\veps\|x\|^2\geq n\}}
  \]
  and, for each natural $j\le \log_4(n/\veps^2)$, we let
  \[
  \Delta_{1,j}:=\veps^{-1} h_j\IND_{\{h_j> \veps^2 2^j\sqrt{|I_j|}\}}.
  \]
  Let $Q_1$ be as in Lemma \ref{le:Q1Q2}. Then $\Delta_1$ satisfies
  \[
  Q_1(\Delta_1)\leq\overline{m}_A(u+\Delta_1)+6\sqrt{\veps}+
  8\veps\sqrt{n}\sqrt{\max_{\ell\in\dN}\frac{|I_\ell|}{16^\ell}}.
  \]
\end{lemma}

\begin{proof}
  Denote by $J$ the set of all $j\le \log_4(n/\veps^2)$ such that $h_j>
  \veps^2 2^j\sqrt{|I_j|}$ and let $J'$ be its complement inside
  $\{1,\dots,\lfloor\log_4(n/\veps^2)\rfloor\}$. Then
  \begin{align*}
    Q_1(\Delta_1)&=\overline{m}_A(u+\Delta_1)
    +\sum_{j=1}^\infty\sum_{i\in I_j}\frac{\langle x,x_i\rangle^2-1}{u+\Delta_1-\lambda_i}\\
    &\le\overline{m}_A(u+\Delta_1) + \sum_{j>\log_4(n/\veps^2)}
    \sum_{i\in I_j}\frac{\langle x,x_i\rangle^2}{u+\Delta_1-\lambda_i}+
    4\sum_{j\in J'}\frac{h_j}{\Delta_1+4^j}
    +4\sum_{j\in J}\frac{h_j}{\Delta_1+4^j}\\
    &=:\overline{m}_A(u+\Delta_1)+ (*)+4(**)+4(***).
  \end{align*}
  We shall estimate the last three quantities separately. First, note that
  \[
  (*)\le \frac{4\|x\|^2}{\Delta_1+n/\veps^2}\le 
  \frac{4\sqrt{\veps}\|x\|^2}{\|x\|^2\IND_{\{\veps\|x\|^2\ge n\}}+n/\veps^{3/2}}
  \le 4\sqrt{\veps}.
  \]
  Next,
  \[
  (**)\le \veps^2\sum_{j\in J'}\sqrt{\frac{|I_j|}{4^j}}
  \le \veps^2\sqrt{\max_{\ell\in\dN}\frac{|I_\ell|}{16^\ell}}\sum_{j\le\log_4(n/\veps^2)}2^j
  \le 2\veps\sqrt{n}\sqrt{\max_{\ell\in\dN}\frac{|I_\ell|}{16^\ell}}.
  \]
  Finally, we have
  \[
  (***)\le \frac{\sum_{j\in J}h_j}{\sum_{j\in J}\Delta_{1,j}}\le\veps.
  \]
  Summing up the estimates, we get the result.
\end{proof}

Denote 
\[
F_2(\Delta):=\Bigl(\sum_{i=1}^n\frac{1}{(u+\Delta-\lambda_i)(u-\lambda_i)}\Bigr)^{-1}
\sum_{i=1}^n\frac{\langle x,x_i\rangle^2}{\bigl(u+\Delta-\lambda_i\bigr)^2},\quad\Delta\ge 0
\]
(clearly, $F_2(\Delta)=\Delta Q_2(\Delta)$ for all $\Delta>0$).

\begin{lemma}[Definition of $\Delta_2$]\label{le:Q2}
  Let $A$, $u$ and $x$ be as above, with $u>\lambda_1$, and let $\alpha>0$ satisfy
  $\overline{m}_A(u)+\alpha<1$. Define $\Delta_2$ according to the following
  procedure: if 
  \[
  (1+\alpha)F_2(0)\le
  \alpha(u-\lambda_1)(1-\overline{m}_A(u)-\alpha)
  \]
  then set
  \[
  \Delta_2:=\frac{(1+\alpha)F_2(0)}{1-\overline{m}_A(u)-\alpha},
  \]
  otherwise, take the smallest non-negative integer $j$ such that
  \[
  Q_2\bigl(2^j(u-\lambda_1)\bigr)\le 1-\overline{m}_A(u)-\alpha
  \]
  and let
  \[
  \Delta_2:=2^j(u-\lambda_1).
  \]
  Then $\Delta_2=0$ whenever $x=0$, and, for $x\neq 0$, we have $\Delta_2\neq 0$ and
  \[
  Q_2(\Delta_2)\le 1-\overline{m}_A(u)-\alpha.
  \]
\end{lemma}

\begin{proof}
  The case $x=0$ is trivial, so further we assume that $x\neq 0$ implying $\Delta_2\neq 0$.
  First, suppose that 
  \[
  (1+\alpha)F_2(0)\leq \alpha(u-\lambda_1)(1-\overline{m}_A(u)-\alpha).
  \]
  Then, in view of the definition of $\Delta_2$, we have $\Delta_2\leq
  \alpha(u-\lambda_1)$, and
  \begin{align*}
    F_2(\Delta_2)%
    &\leq
    \Bigl(\sum_{i=1}^n\frac{1}{(u+\Delta_2-\lambda_i)(u-\lambda_i)}\Bigr)^{-1}
    \sum_{i=1}^n\frac{\langle x,x_i\rangle^2}{\bigl(u-\lambda_i\bigr)^2}\\
    &\leq (1+\alpha)F_2(0).
  \end{align*}
  Hence,
  \[
  Q_2(\Delta_2)=F_2(\Delta_2)/\Delta_2\leq 1-\overline{m}_A(u)-\alpha.
  \]
  If 
  \[
  (1+\alpha)F_2(0)> \alpha(u-\lambda_1)(1-\overline{m}_A(u)-\alpha)
  \]
  then the statement follows directly from the definition of $\Delta_2$.
\end{proof}

\subsection{Randomization and control of expectations}

In this subsection, we ``randomize'' Lemma~\ref{le:Q1} and
Lemma~\ref{le:Q2}. Let, as before, $A$ be an $n\times n$ (non-random) positive
semidefinite matrix with eigenvalues $\lambda_1\ge\dots\ge\lambda_n\ge 0$, let
$u>\lambda_1$ satisfy $\overline{m}_A(u)<1$. We define (random) quantities
$\Delta_1$ and $\Delta_2$ as in Lemma~\ref{le:Q1} and Lemma~\ref{le:Q2},
replacing the fixed vector $x$ with a random isotropic vector $X$. In the
following two lemmas, we will estimate the expectations of $\Delta_1$ and
$\Delta_2$.

\begin{lemma}[Control of $\dE\Delta_1$]\label{le:ED1}
  Let $\veps\in(0,1/4]$ and let $A$, $\lambda_i$, $u$ be as above and $I_j$ be defined
  according to \eqref{eq:def level sets}. Assume additionally that
  \[
  \max_{\ell\in\dN}\frac{|I_\ell|}{16^\ell}\le \frac{1}{\veps n}.
  \]
  Further, let $X\in\dR^n$ be an isotropic random vector and $K\ge 1$. Assume
  that for any non-zero orthogonal projection $\proj:\dR^n\to\dR^n$ we have
  \[
  \dP\bigl\{\|\proj(X)\|^2-\rank\proj\ge t\bigr\} %
  \le \frac{g(\rank\proj)\rank\proj}{t^2},\;\;t\geq f(\rank\proj)\rank\proj,
  \]
  where functions $f:\dN\to[0,1]$ and $g:\dN\to[0,K]$ satisfy
  \[
  g(k)\leq \veps^{11/2}%
  \quad\text{and}\quad%
  f(k)\leq \veps^2\;\;\;\;\forall k\geq \veps^{17}n.
  \]
  Define a random variable $\Delta_1$ as in Lemma~\ref{le:Q1} replacing
  the non-random vector $x$ with $X$.
  Then
  \[
  \dE\Delta_1\leq 32K\sqrt{\veps}.
  \]
\end{lemma}

\begin{proof}
  Let us estimate separately the expectations of $\Delta_1'$ and
  $\Delta_{1,j}$, ($j\leq \log_4(n/\veps^2)$), where the quantities are defined as in
  Lemma~\ref{le:Q1}. First,
  \begin{align*}
    \dE \Delta_1' %
    &=\veps^{-3/2}n\dP\{\|X\|^2\ge n/\veps\} %
    +\int_{\veps^{-3/2}n}^\infty \dP\{\|X\|^2\geq\sqrt{\veps}\tau\}\,d\tau\\
    &\leq\frac{\veps^{4}}{(\veps^{-1}-1)^2}
    +\int_{\veps^{-3/2}n-\veps^{-1/2}n}^\infty \frac{\veps^{9/2}n}{\tau^2}\,d\tau\\
    &\leq \frac{\veps^{4}}{(\veps^{-1}-1)^2}+\frac{\varepsilon^5}{\varepsilon^{-1}-1}\\
    &\leq 4\veps^6.
  \end{align*}
  Next, fix any natural $j\le\log_4(n/\veps^2)$ with $I_j\neq\varnothing$. We
  let $h_j$ to be defined as in \eqref{eq:def hj}, with the random vector $X$
  replacing the non-random $x$. We have
  \[
  \dE\Delta_{1,j}%
  =\veps 2^j \sqrt{|I_j|}\dP\bigl\{h_j> \veps^2 2^j\sqrt{|I_j|}\bigr\}%
  +\int_{\veps 2^j\sqrt{|I_j|}}^\infty%
  \dP\{h_j\ge \veps\tau\}\,d\tau.
  \]
  Note that $h_j=\|\proj_j(X)\|^2-\rank\proj_j$ and $|I_j|=\rank\proj_j$,
  where $\proj_j$ is the orthogonal projection onto the span of $\{x_i,\,i\in
  I_j\}$. Let us consider two cases.

  \noindent 1) $|I_j|\geq \veps^{17}n$. Since $\overline{m}_A(u)<1$,
  we have $|I_j|\leq 4^j$, and $\veps^2 |I_j|\leq
  \veps^2 2^j\sqrt{|I_j|}$. Hence, from the above formula for $\dE\Delta_{1,j}$ and
  the conditions on projections of $X$ we obtain
  \begin{align*}
    \dE\Delta_{1,j}&\leq \frac{\veps^{5/2}\sqrt{|I_j|}}{2^j}
    +\frac{\veps^{5/2}\sqrt{|I_j|}}{2^j}\\
    &\leq 2^{j+1}\veps^{5/2}\sqrt{\max_{\ell\in\dN}\frac{|I_\ell|}{16^\ell}}\\
    &\leq 2^{j+1}\veps^{2} n^{-1/2}.
  \end{align*}
  
  \noindent 2) $|I_j|< \veps^{17}n$. Then, in view of Lemma~\ref{le:prop Ij},
  $|I_j|\leq \veps^8 4^j$, implying
  \[
  f(\rank\proj_j)\leq 1 %
  \leq |I_j|%
  =\sqrt{|I_j|}\sqrt{|I_j|}%
  \leq \veps^4 2^j\sqrt{|I_j|}%
  <\veps^2 2^j\sqrt{|I_j|}.
  \]
  Applying again the formula for $\dE\Delta_{1,j}$ together with the projection conditions
  and Lem\-ma~\ref{le:prop Ij}, we obtain
  \[
  \dE\Delta_{1,j}\le\frac{2K\sqrt{|I_j|}}{\veps^3 2^j}\leq \min\Bigl(
  \frac{2^{j+1}K}{\veps^{7/2}\sqrt{n}},2\veps K\Bigr).
  \]
  Thus, in any case $\dE\Delta_{1,j}\le 2^{j+1}\veps^{2} n^{-1/2}+\min\Bigl(
  \frac{2^{j+1}K}{\veps^{7/2}\sqrt{n}},2\veps K\Bigr)$.
  Summing over all $j\leq \log_4(n/\veps^2)$, we get
  \begin{align*}
  \sum_{j\le\log_4(n/\veps^2)}\dE\Delta_{1,j}&\leq 4 \veps+
  \sum_{j\le\log_4(\veps^9 n)}\frac{2^{j+1}K}{\veps^{7/2}\sqrt{n}}+2K\veps(-11\log_4\veps+1)\\
  &\leq 4\veps+4K\veps+22K\sqrt{\veps}+2K\veps,
  \end{align*}
  and the statement follows.
\end{proof}

The next lemma does not require assumptions on projections of $X$ other than of rank one.

\begin{lemma}[Control of $\dE\Delta_2$]\label{le:ED2}
  Let $A$ and $\lambda_i$ ($i\leq n$) be as above and
  let $X$ be an isotropic random vector in $\dR^n$ such that for some
  $\kappa,K>0$ we have $\max_{\NRM{y}=1}\dE|\langle X,y\rangle|^{2+\kappa}\le
  K$, and let $u>\lambda_1$ and $\alpha\in(0,1/2]$ satisfy
  $\overline{m}_A(u)+\alpha<1$. Then for
  \[
  w_{\ref{le:ED2}}:=K2^{1+\kappa}/(\alpha^{1+\kappa/2}(1-2^{-\kappa/2}))
  \]
  we have
  \[
  \dE\Delta_2\le \frac{1+\alpha}{1-\overline{m}_A(u)-\alpha}%
  \Bigl(1+%
  \frac{w_{\ref{le:ED2}}}{(1-\overline{m}_A(u)-\alpha)^{\kappa/2}(u-\lambda_1)^{\kappa/2}}\Bigr),
  \]
  where $\Delta_2$ is defined as in Lemma~\ref{le:Q2}, with $X$ replacing $x$.
\end{lemma}

\begin{proof}
  An inspection of the definition of $\Delta_2$ in Lemma~\ref{le:Q2}
  immediately shows that
  \[
  \Delta_2\le \frac{(1+\alpha)F_2(0)}{1-\overline{m}_A(u)-\alpha}+\sum_{j=0}^\infty
  2^j(u-\lambda_1)\chi_j,
  \]
  where $\chi_0$ is the indicator function of the event
  \[
  \bigl\{(1+\alpha)F_2(0)> \alpha(u-\lambda_1)(1-\overline{m}_A(u)-\alpha)\bigr\}
  \]
  and for each $j>0$, $\chi_j$ is the indicator of the event
  \[
  \bigl\{Q_2\bigl(2^{j-1}(u-\lambda_1)\bigr)> 1-\overline{m}_A(u)-\alpha\bigr\}.
  \]
  Note that for each fixed number $r\ge 0$, $F_2(r)$ is a random variable
  representable in the form
  \[
  F_2(r)=\sum_{i=1}^n \beta_i\langle X,x_i\rangle^2,
  \]
  where the numbers $\beta_i=\beta_i(r)$ are positive and
  $\sum_{i=1}^n\beta_i\le 1$. By the Minkowski inequality, for any $p\ge 1$
  we have
  \[
  \bigl(\dE F_2(r)^{p}\bigr)^{1/p} %
  \le\sum_{i=1}^n \bigl(\dE(\beta_i\langle X,x_i\rangle^2)^{p}\bigr)^{1/p} %
  =\sum_{i=1}^n \beta_i\bigl(\dE|\langle X,x_i\rangle|^{2p}\bigr)^{1/p} %
  \le \max_{\NRM{y}=1}\bigl(\dE|\langle X,y\rangle|^{2p}\bigr)^{1/p}.
  \]
  In particular, $\dE F_2(r)^{1+\kappa/2}\le \max_{\NRM{y}=1}\dE|\langle
  X,y\rangle|^{2+\kappa}\le K$. Hence, from the definition of the indicator
  functions $\chi_j$ and applying Markov's inequality, we get
  \[
  \dE\chi_0 %
  \le %
  K\bigl(\alpha(1+\alpha)^{-1}(u-\lambda_1)(1-\overline{m}_A(u)-\alpha)\bigr)^{-1-\kappa/2},
  \]
  and for every $j\ge 1$:
  \[
  \dE\chi_j %
  \le %
  K\bigl(2^{j-1}(u-\lambda_1)(1-\overline{m}_A(u)-\alpha)\bigr)^{-1-\kappa/2}.
  \]
  Finally, we obtain
  \begin{align*}
    \dE\Delta_2&\le \frac{1+\alpha}{1-\overline{m}_A(u)-\alpha}\\
    &\qquad+K(u-\lambda_1)\bigl(\alpha(1+\alpha)^{-1}(u-\lambda_1)(1-\overline{m}_A(u)-\alpha)\bigr)^{-1-\kappa/2}\\
    &\qquad+
    \sum_{j=1}^\infty 2^jK(u-\lambda_1)%
    \bigl(2^{j-1}(u-\lambda_1)(1-\overline{m}_A(u)-\alpha)\bigr)^{-1-\kappa/2}\\
    &= \frac{1+\alpha}{1-\overline{m}_A(u)-\alpha}\\
    &\qquad+\frac{K}
    {(1-\overline{m}_A(u)-\alpha)^{1+\kappa/2}(u-\lambda_1)^{\kappa/2}}
    \Bigl(\frac{(1+\alpha)^{1+\kappa/2}}{\alpha^{1+\kappa/2}}
    +\sum_{j=1}^\infty 2^{1+\kappa/2-\kappa j/2}\Bigr)\\
    &\le \frac{(1+\alpha)}{1-\overline{m}_A(u)-\alpha}
    \Bigl(1+\frac{K(1+\alpha)^{\kappa/2}}{\alpha^{1+\kappa/2}
      (1-\overline{m}_A(u)-\alpha)^{\kappa/2}(u-\lambda_1)^{\kappa/2}}
    \sum_{j=0}^\infty 2^{1+\kappa/2-\kappa j/2}\Bigr),
  \end{align*}
  and the result follows.
\end{proof}

\subsection{Proof of Theorem~\ref{th:max}, completed}\label{se:proof of thmax}

Let $X_n$, $m_n$ and $A_n$, $n\in\dN$, be as in \eqref{eq:Xn}, \eqref{eq:mn},
and \eqref{eq:A} respectively. Without loss of generality, we can assume that
both functions $f$ and $g$ in \eqref{eq:STP} are non-increasing. Denote
\[
\gamma:=\inf_{n\to\infty}\frac{m_n}{n}>0
\]
and fix any $\veps\in(0,1/16]$. We define
$\alpha\in(0,\sqrt{\gamma}/(1+\sqrt{\gamma}))$ as the largest number
satisfying
\[
\frac{1+\alpha}{1-t-\alpha}\leq \frac{1+\veps}{1-t}%
\quad\text{for all}\quad%
t\in\Bigl(0,\frac{1}{1+\sqrt{\gamma}}\Bigr].
\]
Further, let $n_\veps\in\dN$ be such that
\[
g(k)\leq \veps^{11/2}%
\quad\text{and}\quad%
f(k)\leq \veps^2\;\;\;\forall k\geq \veps^{17}n_\veps
\]
and, additionally,
\[
\frac{2(48+192g(1))\alpha^{-3/2}}
{(\sqrt{\gamma}/(1+\sqrt{\gamma}))^{3/2}(\sqrt{\veps
    n_\veps}/4)^{1/2}}\leq \veps
\]
(the last condition will be needed later to simplify the estimate coming from Lemma~\ref{le:ED2}).
From now on, we fix $n\ge n_\veps$. For convenience, we let $m:=m_n$ and
$X^{(1)},\dots,X^{(m)}$ be i.i.d.\ copies of $X_n$. We define
\[
A^{(0)}:=0%
\quad\text{and}\quad%
A^{(k)}=A^{(k-1)}+X^{(k)}\otimes X^{(k)},\quad 1\leq k\leq m, 
\]
so that $A_n=A^{(m)}$. Set
\[
u_0:=n+\sqrt{mn}
\]
and let $u_1,\dots,u_m$ be a collection of \emph{random} numbers defined
inductively as follows:
\[
u_k:=u_{k-1}+\Delta_1^k+\Delta_2^k+\Delta_R^k,
\]
where $\Delta_1^k$ is taken from Lemma~\ref{le:Q1} (with $A^{(k-1)}$, $u_{k-1}$
and $X^{(k)}$ replacing $A$, $u$ and $x$, respectively), $\Delta_2^k$ is defined
as in Lemma~\ref{le:Q2} (again, with $A^{(k-1)}$, $u_{k-1}$ and $X^{(k)}$ taking
place of $A$, $u$ and $x$), and $\Delta_R^k$ is the regularity shift defined by
\[
\Delta_R^k:=\min\Bigl\{\ell\in\dN\cup\{0\}:\,\overline{m}_{A^{(k)}}(u_{k-1}+\Delta_1^k+\Delta_2^k+\ell)
-\overline{m}_{A^{(k)}}(u_{k-1}+\Delta_1^k+\Delta_2^k+\ell+1)\le\frac{1}{2\veps n}\Bigr\}
\]
(the factor ``$2$'' in front of $\veps n$ in the above definition does not carry any special
meaning, and introduced for a purely technical reason).
Note that, by lemmas~\ref{le:Q1} and~\ref{le:Q2}, for any $k\geq 1$ we have
$\overline{m}_{A^{(k)}}(u_{k-1}+\Delta_1^k+\Delta_2^k)\leq \overline{m}_{A^{(k-1)}}(u_{k-1})$,
implying
\[
\overline{m}_{A^{(k)}}(u_k)+\alpha %
\le\overline{m}_{A^{(0)}}(u_0)+\alpha %
=\frac{n}{u_0}+\alpha %
=\frac{n}{n+\sqrt{mn}}+\alpha %
<1\quad\text{(deterministically)}
\]
for all $k\ge 0$. The quantities $\Delta_R^k$ play a role analogous to numbers $\delta_R^k$
from the section dealing with the smallest eigenvalue.
The following lemma gives an estimate of the cumulative impact of the regularity shifts:
\begin{lemma}\label{le:deltaC}
  With $\Delta_R^k$ ($1\leq k\leq m$) defined above, the following holds
  deterministically:
  \[
  \sum_{k=1}^m\Delta_R^k\le 2\veps n.
  \]
\end{lemma}

\begin{proof}
  Take any admissible $k\ge 1$. The definition of $\Delta_R^k$ immediately
  implies that for all $0\le\ell<\Delta_R^k$ we have
  \[
  \overline{m}_{A^{(k)}}(u_{k-1}+\Delta_1^k+\Delta_2^k+\ell)-
  \overline{m}_{A^{(k)}}(u_{k-1}+\Delta_1^k+\Delta_2^k+\ell+1)>\frac{1}{2\veps n}.
  \]
  Hence,
  \begin{align*}
    \overline{m}_{A^{(k)}}(u_k)%
    &=\overline{m}_{A^{(k)}}(u_{k-1}+\Delta_1^k+\Delta_2^k+\Delta_R^k)\\
    &\le\overline{m}_{A^{(k)}}(u_{k-1}+\Delta_1^k+\Delta_2^k)-\frac{\Delta_R^k}{2\veps n}\\
    &\le\overline{m}_{A^{(k-1)}}(u_{k-1})-\frac{\Delta_R^k}{2\veps n}.
  \end{align*}
  Thus, $\overline{m}_{A^{(k)}}(u_k)\le
  \overline{m}_{A^{(k-1)}}(u_{k-1})-\frac{\Delta_R^k}{2\veps n}$ for all $k\ge 1$,
  which, together with the relations $0<\overline{m}_{A^{(k)}}(u_k)< 1$ ($1\leq
  k\leq m$), implies the result.
\end{proof}

Now, fix for a moment any $k\in\{1,\dots,m\}$ and let $\mathcal{F}_{k-1}$ be
the $\sigma$-algebra generated by the vectors $X^{(1)},\dots,X^{(k-1)}$. We
will first estimate the conditional expectation
\[
\dE(\Delta_1^k+\Delta_2^k\,|\,\mathcal{F}_{k-1})
=\dE(\Delta_1^k\,|\,\mathcal{F}_{k-1})+\dE(\Delta_2^k\,|\,\mathcal{F}_{k-1}).
\]
Let the sets $I_j$ ($j\in\dN$) be defined by \eqref{eq:def level sets} with
$A^{(k-1)}$ and $u_{k-1}$ replacing $A$ and $u$, i.e.\
\[
I_j:=\bigl\{i\le n:\, 4^{j-1}\le u_{k-1}-\lambda_i(A^{(k-1)})< 4^j\bigr\},\quad
j\in\dN.
\]
Note that the summand $\Delta^{k-1}_R$
makes sure that
\[
\sum_{j=1}^\infty \frac{|I_j|}{16^j}\leq
\sum_{j=1}^\infty \frac{2|I_j|}{4^j(4^j+1)}
< 2\bigl(\overline{m}_{A^{(k-1)}}(u_{k-1})-
\overline{m}_{A^{(k-1)}}(u_{k-1}+1)\bigr)\le\frac{1}{\veps n}
\]
(when $k=1$, the above estimate holds as well since in that case
$\sum_{j=1}^\infty \frac{|I_j|}{16^j}\leq\frac{16n}{u_0^2}<\frac{1}{\veps n}$).
Applying Lemma~\ref{le:ED1}, we obtain
\[
\bigl\|\dE(\Delta_1^k\,|\,\mathcal{F}_{k-1})\bigr\|_\infty
\le 32g(1)\sqrt{\veps}.
\]

Next, note that the assumptions on $1$-dimensional projections imply that for
any unit vector $y$ we have
\begin{align*}
  \dE|\langle y,X_n\rangle|^{3}%
  &=\int_0^\infty \dP\bigl\{\langle y,X_n\rangle^2\geq\tau^{2/3}\bigr\}\,d\tau\\
  &\leq \sqrt{8}+\int_{\sqrt{8}}^\infty %
  \dP\bigl\{\langle y,X_n\rangle^2-1\geq\tau^{2/3}-1\bigr\}\,d\tau\\
  &\leq \sqrt{8}+4g(1)\int_{\sqrt{8}}^\infty\tau^{-4/3}\,d\tau\\
  &\leq \sqrt{8}+12g(1).
\end{align*}
Next, the upper bound on $\frac{|I_j|}{16^j}$ implies that $I_j=\varnothing$
whenever $j\leq\log_{16}(\veps n)$, whence 
\[
u_{k-1}-\lambda_1(A^{(k-1)})\geq\frac{\sqrt{\veps n}}{4}.
\]
Then, by Lemma~\ref{le:ED2} (applied with $\kappa:=1$), and in view of the inequality
$\overline{m}_{A^{(k-1)}}(u_{k-1})\leq \overline{m}_{A^{(0)}}(u_0)$, we get
\begin{align*}
  \bigl\|%
  &\dE(\Delta_2^k\,|\,\mathcal{F}_{k-1})\bigr\|_\infty\\
  &\le \frac{1+\alpha}{1-\overline{m}_{A^{(0)}}(u_0)-\alpha}
  \Bigl(1+%
  \frac{w_{\ref{le:ED2}}}%
  {(1-\overline{m}_{A^{(0)}}(u_0)-\alpha)^{1/2}(\sqrt{\veps n}/4)^{1/2}}\Bigr),
\end{align*}
where
\[
w_{\ref{le:ED2}}%
=4\sup_{\|y\|=1}\dE|\langle
y,X_n\rangle|^{3}/(\alpha^{3/2}(1-2^{-1/2})) \leq (48+192g(1))\alpha^{-3/2}.
\]
By the choice of $\alpha$, we have
\[
\frac{1+\alpha}{1-\overline{m}_{A^{(0)}}(u_0)-\alpha}%
\leq \frac{1+\veps}{1-\overline{m}_{A^{(0)}}(u_0)},
\]
and by the choice of $n_\veps$, we have
\[
\frac{w_{\ref{le:ED2}}(1+\alpha)}%
{(1-\overline{m}_{A^{(0)}}(u_0)-\alpha)^{3/2}(\sqrt{\veps n}/4)^{1/2}}%
\leq \veps.
\]
Hence, we obtain
\[
\bigl\|\dE(\Delta_2^k\,|\,\mathcal{F}_{k-1})\bigr\|_\infty%
\leq (1+\veps)\Bigl(1+\sqrt{\frac{n}{m}}\Bigr)+\veps.
\]
Summing up over $k\in\{1,\ldots,m\}$ and applying Lemma~\ref{le:deltaC}, we
get
\begin{align*}
  \dE\lambda_1(A_n)
  &\leq \dE u_{m}\\
  &\leq n+\sqrt{mn}+32g(1)\sqrt{\veps}m+(1+\veps)(m+\sqrt{mn})+\veps m+2\veps n\\
  &=(\sqrt{n}+\sqrt{m})^2+32g(1)\sqrt{\veps}m+\veps(2m+2n+\sqrt{mn}).
\end{align*}
Since $\veps$ can be chosen arbitrarily small, we get the result.

\section*{Acknowledgements}

D.C.\ would like to warmly thank Radosław Adamczak, Charles Bordenave, and
Alain Pajor for discussions on this topic in Warsaw, Toulouse, and Paris.
K.T.\ would like to thank Pierre Youssef for introducing him to the result of
Batson--Spielman--Srivastava and especially thank Prof.\ Alain Pajor for the
invitation to visit UPEM in November--December, 2014. A significant part of
the present work was done during that period.

\makeatletter%
\def\@MRExtract#1 #2!{#1}
\renewcommand{\MR}[1]{
  \xdef\@MRSTRIP{\@MRExtract#1 !}%
  \href{http://www.ams.org/mathscinet-getitem?mr=\@MRSTRIP}{MR-\@MRSTRIP}}%
\makeatother%

\newcommand{\etalchar}[1]{$^{#1}$}
\providecommand{\bysame}{\leavevmode\hbox to3em{\hrulefill}\thinspace}
\providecommand{\MR}{\relax\ifhmode\unskip\space\fi MR }
\providecommand{\MRhref}[2]{%
  \href{http://www.ams.org/mathscinet-getitem?mr=#1}{#2}
}
\providecommand{\href}[2]{#2}

\end{document}